\documentclass[letter, left=3cm, right=3cm, top=3cm, bottom=3.5cm]{amsart}
\usepackage{marginnote}
\usepackage{amsmath,amsfonts,amsthm,amssymb}
\usepackage{geometry}
\usepackage{graphicx}
\numberwithin{equation}{section}

\theoremstyle{plain}
\newtheorem{theorem}{Theorem}[section]
\newtheorem{hypothesis}{Hypothesis}
\newtheorem{step}{Step}
\newtheorem{lemma}[theorem]{Lemma}
\newtheorem{proposition}[theorem]{Proposition}

\newtheorem{remark}[theorem]{Remark}

\theoremstyle{definition}
\newtheorem{definition}{Definition}

\DeclareMathOperator{\dd}{d\!}
\DeclareMathOperator{\n}{{\bf n}}

\newcommand{\R}{\mathbb{R}}
\newcommand{\SF}{\mathbb{S}}

\newcommand{\Z}{\mathbb{Z}}

\begin{document}

\title[Multiphase solutions]{ Multiphase solutions to the vector Allen-Cahn equation: crystalline and other complex symmetric structures
}
\author{Peter W. Bates}
\address[P.~Bates]
{Department of Mathematics\\
 Michigan State University\\
 East Lansing, Michigan 48824}
\email[P.~Bates]{bates@math.msu.edu}
\author{Giorgio Fusco}
\address[G.~Fusco]{Dipartimento di Matematica Pura ed Applicata\\ Universit\`a degli Studi dell'Aquila\\ Via Vetoio\\ 67010 Coppito\\ L'Aquila\\ Italy}
\email[G.~Fusco]{fusco@univaq.it}
\author{Panayotis Smyrnelis} \address[P.~ Smyrnelis]{Department of Mathematics\\ University of Athens\\ Panepistemiopolis\\ 15784 Athens\\ Greece}
\email[P. ~Smyrnelis]{smpanos@math.uoa.gr}
\thanks{PWB was supported in part by NSF DMS-0908348,  DMS-1413060, and the IMA.  GF was partially supported by the IMA.  PS was partially supported through the project PDEGE – Partial Differential Equations
Motivated by Geometric Evolution, co-financed by the European Union – European Social
Fund (ESF) and national resources, in the framework of the program Aristeia of the ‘Operational
Program Education and Lifelong Learning’ of the National Strategic Reference Framework
(NSRF)}




\begin{abstract}
΅We present a systematic study of entire symmetric solutions $u:\R^n \to \R^m$ of the vector Allen-Cahn equation $\Delta u-W_u(u)=0$, $x \in \R^n$, where $W:\R^m \to \R$ is smooth, symmetric, nonnegative with a finite number of zeros, and $ W_u := (\partial W / \partial u_1, \dots, \partial W / \partial u_m)^{\top}$. We introduce a general notion of equivariance with respect to a homomorphism $f:G \to \Gamma$ ($G$, $\Gamma$ reflection groups) and prove two abstract results, concerning the cases of $G$ finite and $G$ discrete, for the existence of equivariant solutions. Our approach is variational and based on a mapping property of the parabolic vector Allen-Cahn equation and on a pointwise estimate for vector minimizers.
\end{abstract}

\maketitle
\section{Introduction}
\setcounter{figure}{0}
We study bounded solutions $u:\R^n\rightarrow\R^m$ to the vector Allen-Cahn equation
\begin{eqnarray}\label{elliptic system}
\Delta u-W_u(u)=0,\;x\in\R^n,
\end{eqnarray}
where $W:\R^m\rightarrow\R$ is a smooth nonnegative multi-well potential that vanishes at $N\geq 1$ distinct points $a_1, a_2, \cdots, a_N \in \R^m$, and $ W_u := (\partial W / \partial u_1, \dots, \partial W / \partial u_m)^{\top}$.
The stationary Allen-Cahn equation is formally the $L^2$ Euler-Lagrange equation associated to the functional
\begin{eqnarray}\label{functional}
J_\Omega(u):=\int_\Omega\Big{(}\frac{1}{2}\vert\nabla u\vert^2+W(u) \Big{)} \dd x,
\end{eqnarray}
defined on each bounded domain $\Omega\subset\R^n$ for $u\in W^{1,2}(\Omega,\R^m)$.

A complete description of the set of  entire solutions of (\ref{elliptic system}), even in the scalar case $m=1$ in spite of many interesting and deep results (see for example \cite{dancer} \cite{malchiodi} \cite{delpino-wei}  and the references therein), appears to be an impossible task. In this paper we are mostly interested in the vector case $m>1$. We focus on {\it symmetric} potentials and proceed to a systematic discussions of symmetric solutions that can be determined by minimization of the energy (\ref{functional}).

The restriction to the symmetric setting has both physical and mathematical motivations. From the physical point of view we observe that symmetry is ubiquitous in nature and that
 the modeling of systems (e.g. materials) that can exist in different symmetric crystalline
 phases requires the use of vector {\it order parameters} \cite{bcmw}.  Symmetric structures are observed at the junction of three or four coexisting phases in physical space. Similar
 structures appear at the singularities of soap films and compounds of soap bubbles.

From the mathematical point of view, as we discuss in Sec.2, symmetry plays an essential role in the derivation of pointwise
 estimate for solutions of (\ref{elliptic system}). Indeed by exploiting the symmetry we show,
 as in \cite{alikakos-fusco} and  \cite{alikakos-smyrnelis},  the existence of minimizers of
 $J_\Omega$ that map {\it fundamental domains} in the domain $x$-space into {\it fundamental domains} in the target $u$-space. A basic consequence of this is the
 existence of minimizers that in certain sub-domains avoid all the minima of $W$ but one. It follows that in each such sub-domain the potential $W$ can be considered to have a unique
 global minimum. This is a key point since, for potentials with two or more global minima,
 it is very difficult to determine {\it a priori} in which subregions a minimizer $u$ of
 $J_\Omega$ is near to one or another of the minima of $W$. Even in the scalar case,
 $m=1$, for a potential with two minima this has proved to be a very difficult question. For
 instance, this is one of the main difficulties  in the problem of characterizing the level sets of solutions $u:\R^n\rightarrow\R$ of the equation
\begin{eqnarray}\label{dg}
\Delta u=u^3-u,\;x\in\R^n,
\end{eqnarray}
that satisfy the bound $\vert u\vert\leq 1$ and the conditions
\begin{eqnarray}\label{conditions}
\frac{\partial u}{\partial x_n}> 0,
\end{eqnarray}
a problem related to a celebrated conjecture of De Giorgi \cite{farina-valdinoci}.

The study of equation (\ref{elliptic system}) under symmetry hypotheses on the potential $W$ was initiated in \cite{bronsard-reitich} and  \cite{bronsard-gui-schatzman} where the existence of a solution $u:\R^2\rightarrow\R^2$ of (\ref{elliptic system}) with the
 symmetry  of the equilateral triangle was established. Existence of solutions with the
 symmetry of the regular tetrahedron was later proved in \cite{gui-schatzman} for $n=m=3$. These solutions are known as the triple and the quadruple {\it junction} states,
 respectively and, as indicated before with reference to soap films, are related to minimal surface complexes and in particular to the local structure of singularities where three sheets meet along a line or where four of these lines meet at a point \cite{taylor}. Existence of solutions $u:\R^n\rightarrow\R^m$ equivariant with respect to a generic finite reflection group $G$ acting both on the domain and on the target space was studied in \cite{alikakos-fusco}, \cite{alikakos} and \cite{fusco}.

We introduce an abstract notion of equivariance of maps that includes as a special case the notion considered in  \cite{bronsard-gui-schatzman}, \cite{gui-schatzman} and \cite{alikakos-fusco}. We assume there is a finite or discrete (infinite) reflection group $G$ acting on $\R^n$ and a finite reflection group $\Gamma$ acting on $\R^m$ and assume there exists a homomorphism $f:G\rightarrow\Gamma$ between $G$ and $\Gamma$. For the concept of fundamental domain and  for the general theory of reflection groups we refer to \cite{grove-benson} and \cite{humphreys}. We define a map $u:\R^n\rightarrow\R^m$ to be $f$-equivariant if
\begin{eqnarray}\label{f-equivariant}
u(g x)=f(g)u(x),\;\text{ for }\;g\in G,\;x\in\R^n.
\end{eqnarray}
We characterize the homomorphisms which allow for the existence of $f$-equivariant maps that map a fundamental domain $F$ for the action of $G$ on $\R^n$ into a fundamental domain $\Phi$ for the action of $\Gamma$ on $\R^m$:
\begin{eqnarray}\label{positivity}
u(\overline{F})\subset\overline{\Phi}.
\end{eqnarray}
We refer to these homomorphism and to the maps that satisfy (\ref{positivity}) as {\it positive}. Positive homomorphisms (see Definition \ref{positive homomorphism} below) have certain mapping properties that relate the projections associated to the walls of a fundamental domain $F$ to the projections associated to the walls of a correponding region $\Phi$. These properties are instrumental to show that minimizing in the class of $f$-equivariant maps that satisfy (\ref{positivity})
does not affect the Euler-Lagrange equation and renders a smooth solution of (\ref{elliptic system}). The proof of this fact is based on a quite sophisticated use of the maximum principle for parabolic equations that was first introduced in \cite{smyrnelis} and \cite{alikakos-smyrnelis}. We prove (see Lemma \ref{positivity flow}) that, provided $f$ is a positive homomorphism, the $L^2$ gradient flow associated to the functional \eqref{functional} preserves the positivity condition \eqref{positivity}. By a careful choice of certain scalar projections of the vector parabolic equation that describes the above mentioned gradient flow, we show that this fact is indeed a consequence of the maximum principle. Based on this and on a pointwise estimate from \cite{fusco} and \cite{f}  we prove two abstract existence results: Theorem \ref{theorem1} which concerns the case where $G$ is a finite reflection group and Theorem \ref{theorem2} that treats the case of a discrete (infinite) group $G$.

From (\ref{positivity}) and the $f$-equivariance of $u$ it follows
\begin{eqnarray}\label{positivity-1}
u(g\overline{F})\subset f(g)\overline{\Phi},\;\text{ for }\;g\in G.
\end{eqnarray}
Therefore, besides its importance for the proofs of Theorem \ref{theorem1} and Theorem \ref{theorem2}, the mapping property (\ref{positivity}) is a source of information on the
 geometric structure of the vector valued map $u$. The fact that (\ref{positivity-1}) holds
 true in general in the abstract setting of the present analysis can perhaps be regarded as a
 major result of our work. Indeed due to the variety of  choices for $n$ and $m$, the
 dimensions of domain and target space, of the possible choices of the reflection groups $G$ and $\Gamma$, and of the homomorphism $f:G\rightarrow\Gamma$, we will deduce
 from Theorem \ref{theorem1} and Theorem \ref{theorem2} the existence of various
 complex multi-phase solutions of (\ref{elliptic system}) including several types of lattice
 solutions.  A characterization of all homomorphisms between reflection groups in general
 dimensions is not known. For the special case $n=m=2$, in the Appendix, we determine
 all positive homomorphisms between finite reflection groups and the corresponding solutions of \eqref{elliptic system}.

The paper is organized as follows. In Sec.2 we discuss the notion of $f$-equivariance, the
 concept of {\it positive} homomorphism $f:G\rightarrow\Gamma$ and give a few
 examples. In Sec.3 we list the hypotheses on the potential $W$ and state the main results,
 Theorem \ref{theorem1} and Theorem \ref{theorem2}, that we prove in Sec.4. In Sec.5 and
 in Sec.6 we apply the abstract theorems proved in Sec.4 to specific situations and present a series of solutions of (\ref{elliptic system}).

We denote by $\langle x,y\rangle$ the standard inner product of $x, y\in\R^k,\;k\geq 1$, by $\vert x\vert=\sqrt{\langle x,x\rangle}$ the norm of $x\in\R^k$ and by $d(x,A)=\inf_{y\in A}\vert y-x\vert$ the distance of $x$ from $A\subset\R^k$.

\section{$f$-Equivariance and the notion of positive homomorphism}\label{sec-2}

Here and in the following section, together with the abstract concepts and general proofs,
 we present simple but significant examples and special cases to help the reader to build a
 deep understanding of the paper.
We begin with some  examples of $f$-equivariant maps. We let $I_k$ the identity map of $\R^k,\;k\geq 1$. As a first example we observe
 that, in the particular case where $\Gamma=G$ and the homomorphism $f$ is the
 identity, $f$-equivariance reduces to the notion considered in \cite{bronsard-gui-schatzman}, \cite{gui-schatzman}, \cite{alikakos-fusco}, \cite{alikakos}, and \cite{fusco}:
\[u(g x)=g u(x),\;\text{ for }\;g\in G,\;x\in\R^n.\]
The next example is a genuine $f$-equivariant map.
In \cite{bates-fusco}, under the assumption that $W$ is invariant under the group $\Gamma$ of the equilateral triangle, we constructed a solution $ u: \R^2 \to \R^2$ to system \eqref{elliptic system} such that
\begin{itemize}
\item[(i)] $u(\gamma x) = \gamma u(x)$ for all $\gamma\in\Gamma$  (which is the dihedral group $D_3$).
\item[(ii)] $u(-x)=u(x)$ for all $x\in \R^2$.
\end{itemize}
If we incorporate the additional symmetry $(ii)$ in a group structure, this solution can be
 seen as an $f$-equivariant map. Indeed, the regular hexagon reflection group $G=D_6$
 contains $\Gamma=D_3$, and the antipodal map $\sigma:\R^2 \to \R^2$ given by
 $\sigma(x)=-x$. Since $\sigma$ commutes with the elements of $D_3$, $G$ is
 isomorphic to the group product $\{I_2,\sigma \} \times D_3$. Furthemore, we can define a homomorphism $f: D_6=\{I_2
,\sigma \} \times D_3 \to D_3$, by setting $f(\gamma)=\gamma$ and $f( \sigma \gamma)=\gamma$, for every $\gamma \in D_3$. Then, the above conditions $(i)$ and $(ii)$ express the $f$-equivariance of the solution $u$ in \cite{bates-fusco}.

Similarly, we can consider the action on $\R^2$ of the discrete reflection group $G'$
 generated by the reflections $s_1$, $s_2$ and $s_3$ with respect to the correponding
 lines $P_1:=\{ x_2=0 \}$, $P_2:=\{x_2=x_1/ \sqrt 3 \}$ and $P_3:=\{x_2=-\sqrt 3
 (x_1-1)\}$ (the dashed lines in Figure \ref{figure1}). These three lines bound a triangle
 $F'$ with angles 30, 60 and 90 degrees, which is a fundamental domain of $G'$. The
 discrete group $G'$ contains also all the reflections with respect to the lines drawn in Figure \ref{figure1}, which partition the plane into triangles congruent to $F'$.

\begin{figure}[h]
\begin{center}
\includegraphics[scale=.7]{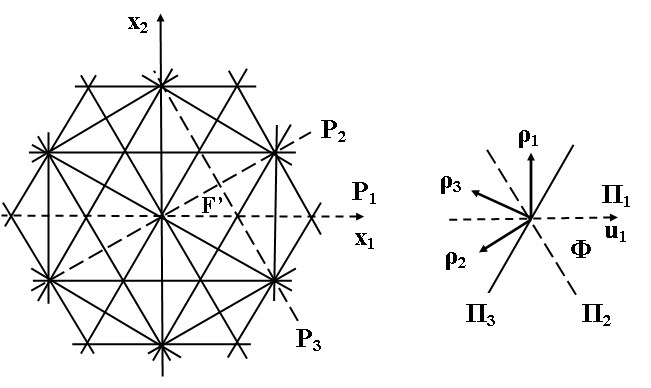}
\end{center}
\caption{The discrete reflection group $G'$ on the left and the finite reflection group $\Gamma=D_3$ on the right.}
\label{figure1}
\end{figure}

 The point group of $G'$, that is the stabilizer of the origin: $\{g\in G^\prime:g(0)=0\}$, is
 the group $G=D_6$, and we have $G'=TG$, where $T$ is the translation group of $G'$.
$T$ is generated by the translations $t^{\pm}$ by the vectors $(\frac{3}{2}, \pm \frac{\sqrt 3}{2})$. Now, if we compose the canonical homomomorphism $p: G' \to G$ such that
 $p(tg)=g$ for every $t \in T$ and $g \in G$, with the homomorphism $f: D_6 \to D_3$
 defined in the previous paragraph, we obtain a homomorphism $f': G' \to D_3$. We have in particular
\begin{equation}\label{f-pos}
\begin{split}
& f^\prime(s_1)=f(p(s_1))=s_1,\\
& f^\prime(s_3)=f(p(s_3))=p(s_3),\\
& f^\prime(s_2)=f^\prime(\sigma p(s_3))=f(\sigma p(s_3))=p(s_3),\\
\end{split}
\end{equation}
where $p(s_3)$ is the reflection in the line $\Pi_2=\{u_2=-\sqrt{3} u_1\}$. We note that the image of the line $P_1=\{x_2=0 \}$ by an $f'$-equivariant map $u:\R^2 \to \R^2$ is
 contained in the line $\Pi_1:=\{u_2=0 \}$ while the images of the lines   $P_2=\{x_2=x_1/
 \sqrt 3 \}$ and $P_3=\{ x_2=-\sqrt 3 (x_1-1) \}$ are contained in the line  $\Pi_2:=\{u_2=-\sqrt 3 u_1 \}$. Indeed
\begin{equation}
\begin{split}
&x=s_1 x\quad\Rightarrow\quad u(x)=u(s_1 x)= f^\prime(s_1)u(x)=s_1u(x),\\
&x=s_2 x\quad\Rightarrow\quad u(x)=u(s_2 x)= f^\prime(s_2)u(x)=p(s_3)u(x),\\
&x=s_3 x\quad\Rightarrow\quad u(x)=u(s_3 x)= f^\prime(s_3)u(x)=p(s_3)u(x).\\
\end{split}
\end{equation}
The lines $\Pi_1$ and $\Pi_2$ define a 60 degree sector $\Phi$ which is a fundamental domain of the finite reflection group $D_3$. At a later stage, we will prove the existence of a solution to  \eqref{elliptic system} that maps the triangle $F'$ in this sector.

 Now we return to the general setting and discuss the notion of positive homomorphism $f:G\to\Gamma$ between reflection groups $G$ and $\Gamma$. Before giving the
 definition we observe that if $s\in G$ is a reflection we have $I_m=f(I_n)=f(s)f(s)$. This and the fact that $f(s)$ is an orthogonal
 transformation imply that $f(s)$ is symmetric. Thus $f(s)$ has $m$ ortho-normal eigenvectors $\nu_1,\dots,\nu_m$ and
 $\nu_j=f(s)f(s)\nu_j=\mu_j^2\nu_j$ implies that $\vert \mu_j\vert=1$ for the
 corresponding eigenvalues $\mu_j,\;j=1,\dots,m$. Therefore if we let $E\subset\R^m$ be
 the span of the eigenvectors corresponding to the eigenvalue $\mu=1$, that is,
 $E=\ker(f(s)-I_m)$ the set of the points fixed by $f(s)$, we have $\R^m=E\oplus E^\bot$ and
  \begin{equation}\label{generalized-projection}
  f(s)u=f(s)(u_E+(u-u_E))=u_E-(u-u_E)=-u+2u_E,
 \end{equation}
 where we have used the decomposition $u=u_E+(u-u_E)$ with $u_E\in E$ and $u-u_E\in E^\bot$. We can interpret  (\ref{generalized-projection}) by saying that $f(s)$ is a
 projection with respect to the subspace $E$ or that $f(s)$ coincides with $I_m$ on $E$ and
 with the antipodal map on $E^\bot$.


\begin{definition}\label{positive homomorphism}
Let $F$ be a fundamental domain of $G$, bounded by the hyperplanes $P_1, \ldots ,P_l$, correponding to the reflections $s_1, \ldots, s_l$. We say that a homomorphism $f: G \to
 \Gamma$ is {\em positive} if there exists a fundamental domain $\Phi$ of $\Gamma$,
 bounded by the hyperplanes $\Pi_1,\ldots, \Pi_k$, such that for every $i=1, \ldots,l$,
 there is $1\leq k_i\leq k$  and $\tilde{\Pi}_1,\ldots,\tilde{\Pi}_{k_i}\in\{\Pi_1,\ldots, \Pi_k\}$ such that
\begin{equation}\label{positive-hom}
\ker(f(s_i)-I_m)=\cap_{j=1}^{k_i}\tilde{\Pi}_j.
\end{equation}
That is, the set of points fixed by the orthogonal map $f(s_i)$ is one of the hyperplanes $\Pi_j$, or the intersection of several of them.
\end{definition}

The property of being positive for a homomorphism  $f$ is independent of the choice of $F$. Indeed, if we take $\hat{F}=gF$, with $g \in G$, then $\hat{F}$
is bounded by the hyperplanes  $gP_1, \ldots ,gP_l$, correponding to the reflections $gs_1g^{-1}, \ldots, gs_lg^{-1}$. In addition $\ker(f(gs_ig^{-1})-I_m)=f(g) \ker(f(s_i)
-I_m)$, thus, the fundamental domain $\hat{\Phi}=f(g)\Phi$ can be associated with
 $\hat{F}$, in accordance with the definition.

Note that the choice of $\Phi$ is not unique, since the homomorphism $f$ can associate $F$ to $\Phi$, or to $-\Phi$ indifferently.

 The homomorphism $f': G' \to D_3$ defined above is an example of a positive homomorphism. Indeed, if we identify $F$ with the triangle $F^\prime$ and $\Phi$ with
 the $60$ degree sector $\{-\sqrt{3}u_1<u_2<0,\;u_1>0\}$ bounded by the lines $\Pi_1$ and $\Pi_2$, then (\ref{f-pos}) expresses the positivity of $f^\prime$.

  It is not true in general that a homomorphism $f:G\rightarrow\Gamma$ between reflection groups $G$ and $\Gamma$ is positive. For example the canonical projection $p$
 of a discrete reflection group $G'$ onto its point group $G$ does not, in general, fulfill this
 requirement. To see this, let us revisit the discrete reflection group $G'$ depicted in Figure \ref{figure1}. We have
   \[\begin{split}
  &p(s_1)=\text{ reflection in the line }\;\Pi_1,\\
  &p(s_2)=\text{ reflection in the line }\;\{u_2=u_1/\sqrt{3}\},\\
  &p(s_3)=f(s_3)=\text{ reflection in the line }\;\{u_2=-\sqrt{3}u_1\},
  \end{split}\]
   then $p(s_i)$,\,$i=1,2,3$ are reflections with respect to three distinct lines intersecting at the origin. Thus, the canonical projection $p: G' \to G=D_6$ cannot associate $F'$ to any fundamental domain of $D_6$ (a 30 degree sector).

\section{The theorems}
We assume:
\begin{hypothesis}[Homomorphism]\label{h1}
 There exist: a finite (or discrete) reflection group $G$  acting on $\R^n$, a finite reflection group $\Gamma$ acting on $\R^m$, and a homomorphism $f : G \to \Gamma$ which
is positive in the sense of Definition \ref{positive homomorphism}. We denote by $\Phi$  the fundamental domain of $\Gamma$
that $f$ associates with the fundamental domain $F$ of $G$.
\end{hypothesis}

\begin{hypothesis}[]\label{h2}
The potential $W: \R^m \to [0,\infty)$, of class $C^3$, is invariant under the finite reflection group $\Gamma$, that is,
\begin{equation}\label{g-invariance}
W(\gamma u) = W(u), \text{ for all } \gamma \in \Gamma \text{ and } u \in \R^m.
\end{equation}
Moreover, we assume that there exists $M>0$ such that
$W(su) \geq W(u)$, for $s\geq 1$ and $|u|=M.$
\end{hypothesis}

\begin{hypothesis}[ ]\label{h3}
There exist $a\in\overline{\Phi}$,  the closure of $\Phi$, and $q^*>0$ such that:
\begin{enumerate}
\item $0=W(a)<W(u),\;\text{ for }\; u\in\overline{\Phi}\setminus\{a\}$ and
\item for each $\nu\in\SF^{m-1}$, $\SF^{m-1}\subset\R^m$ the unit sphere, the map
$(0,q^*]\ni q\rightarrow W(a+q\nu)$ has a strictly positive second derivative.
\end{enumerate}.
\end{hypothesis}

Hypothesis \ref{h2} and \ref{h3} determine the number $N$ of  minima of $W$. From Hypothesis \ref{h2} we have
\[W(\gamma a)=0,\;\text{ for }\;\gamma\in\Gamma.\]
Therefore, if $a\in\Phi$, that is, $a$ is in the interior of $\Phi$, from the fact that $\gamma\Phi\neq\Phi$ for $\gamma\in\Gamma\setminus\{I_m\}$ it follows that $W$ has
 exactly $N=\vert\Gamma\vert$ distinct minima, where $\vert\mathcal{G}\vert$ denotes
 the order of a group $\mathcal{G}$. If $a\in\partial\Phi$ then the stabilizer
 $\Gamma_a=\{\gamma\in\Gamma:\gamma a=a\}$ of $a$ is nontrivial and we have
 $N=\vert\Gamma\vert/\vert\Gamma_a\vert<\vert\Gamma\vert$ and $a$ is the unique minimum of $W$ in the cone $\mathcal{D}\subset\R^m$ defined by
\[\mathcal{D}=\text{ Int }\,\cup_{\gamma\in\Gamma_a}\gamma\overline{\Phi}.\]
The set $\mathcal{D}$ satisfies
\begin{equation}\label{d-property}
\text{ for }\;\gamma\in\Gamma:\;\text{ either }\;\gamma\mathcal{D}\cap\mathcal{D}=\varnothing\;\text{ or }\;\gamma\mathcal{D}=\mathcal{D}.
\end{equation}
It follows that
\[\R^m=\cup_{\gamma\in\Gamma}\gamma\overline{\mathcal{D}},\]
that is, $\R^m$ is partitioned into $N=\vert\Gamma\vert/\vert\Gamma_a\vert$ cones congruent to $\mathcal{D}$.
The cone $\mathcal{D}\subset\R^m$ has its counterpart in the set $D\subset\R^n$ given by
\begin{eqnarray}\label{d-set}
D= \mathrm{Int}\left( {\cup_{g\in f^{-1}(\Gamma_{a})} g \overline{F}} \right)
\end{eqnarray}
which is mapped into $\overline{\mathcal{D}}$ by any positive $f$-equivariant map $u
:\R^n\to\R^m$.  Indeed, for such a map (\ref{positivity-1}) implies that
 $u(g\overline{F})\subset\overline{\mathcal{D}}$ if and only if $f(g)\in\Gamma_a$ or equivalently $g\in f^{-1}(\Gamma_a)$.



In the case   $G=\Gamma$ and $f=I$, the expression for $D$ reduces to $D= \mathrm{Int}\left( {\cup_{g\in \Gamma_{a}} g \overline{F}} \right),$  considered in \cite{alikakos-fusco}, \cite{alikakos}, and \cite{fusco}.

The set $D$ satisfies the analog of (\ref{d-property}). Therefore also the domain space $\R^n$ is partitioned into $N=\vert\Gamma\vert/\vert\Gamma_a\vert$ sets congruent to $D$.

Let us consider a few examples
\begin{enumerate}
\item if $G=\Gamma$, $f$ is the identity and $a$ is in the interior of $\Phi$, then $\mathcal{D}=\Phi$ and $D=F$.
\item if $m=n=2$, $G=\Gamma=D_3$, $\Phi=\{u: 0<u_2<\sqrt{3}u_1,\;u_1>0\}$, $f$ is
 the identity and $a=(1,0)$, then $\Gamma_a=\{I_2,g_1\}$ where $g_1$ is the reflection in
 the line $\{u_2=0\}$. Therefore (\ref{d-set}) yields $D=\mathrm{Int}\left(\overline{F}\cup g_1\overline{F}\right)=\{x:\vert x_2\vert<\sqrt{3}x_1,\;x_1>0\}$.
\item if $n=m=2$, $G=D_6=\{I_2,\sigma\}\times D_3,\;\Gamma=D_3$, $f(\gamma)=\gamma$ and $f(\sigma\gamma)=\gamma$ for every $\gamma\in D_3$,
 and $a\in\Phi=\{u: 0<u_2<\sqrt{3}u_1,\;u_1>0\}$, then $\Gamma_a=\{I_2\}$ and $f^{-1
}(\Gamma_a)=\{I_2,\sigma\}$. Therefore $\mathcal{D}=\Phi$ and
 $D=\mathrm{Int}\left(\overline{F}\cup \sigma\overline{F}\right)=\{x:0<-\frac{x_1x_2
}{\vert x_1\vert}<\frac{\vert x_1\vert}{\sqrt{3}},\; x_1\neq 0\}$. ($F=\{x:0<-x_2<\frac{x_1}{\sqrt{3}}, x_1>0\}$)
\item If in the previous example we take $a=(1,0)\in\overline{\Phi}$ we have
 $\Gamma_a=\{I_2,g_1\}$ and  $f^{-1}(\Gamma_a)=\{I_2,\sigma,g_1,\sigma g_1\}$. It
 follows $\mathcal{D}=\mathrm{Int}\left(\overline{\Phi}\cup g_1\overline{\Phi}\right)=\{0<\vert u_2\vert<\sqrt{3}u_1,\;u_1>0\}$ and
    $D=\mathrm{Int}\left(\overline{F}\cup \sigma\overline{F}\cup g_1\overline{F} \cup
 \sigma g_1 \overline{F} \right)=\{x:0< \vert x_2 \vert<\frac{\vert x_1\vert}{\sqrt{3}} \}$.
\end{enumerate}

​

\begin{figure}[h]

\begin{center}

\includegraphics[scale=0.44]{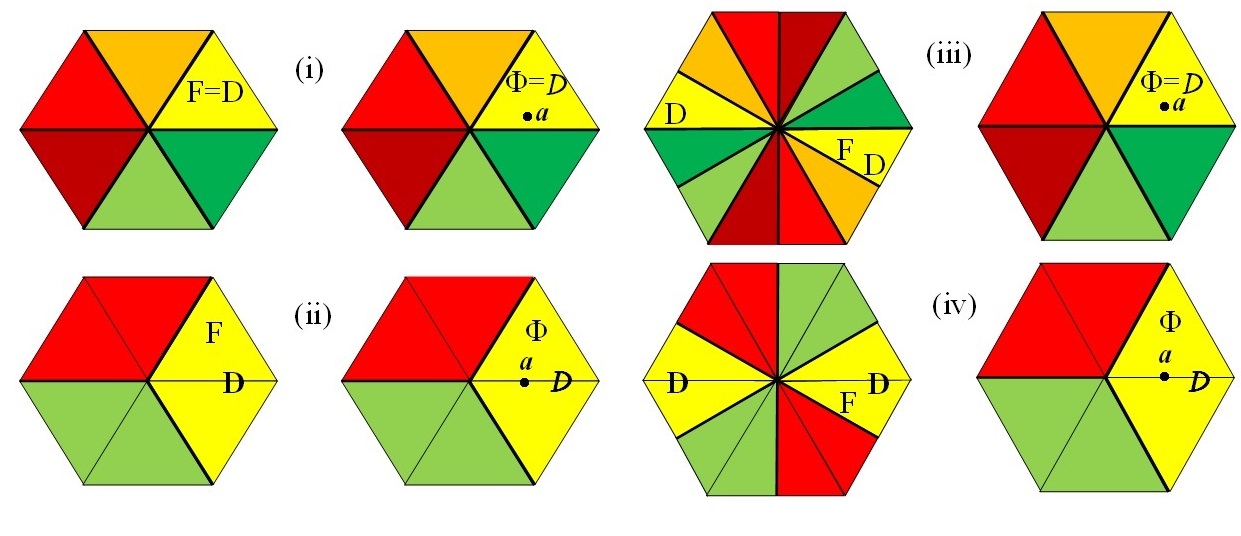}

\end{center}

\caption{The sets $F$, $\Phi$, $D$ and $\mathcal{D}$ and their correspondence by an $f$-equivariant map in the the examples (i)-(iv).}

\label{figurenew}

\end{figure}


If $G$ is a discrete (infinite) group, then $D$ has infinitely many connected components.
As examples (iii) and (iv) above show, even when $G$ is a finite group, $D$ does not need
 to be connected. To characterize one of the connected components of $D$, let
 $G^a\subset f^{-1}(\Gamma_a)$ be the subgroup generated by $f^{-1}(\Gamma_a)\cap\{s_1,\dots,s_l\}$ and define
\begin{eqnarray}\label{d0-set}
D_0:= \mathrm{Int}\cup_{g\in G^a} g \overline{F}.
\end{eqnarray}
Since $G^a$ is a reflection group generated by a subset of $\{s_1,\dots,s_l\}$, the set of
 reflections in the planes that bound $F$, the set $D_0$ is connected. To show that $D_0$
 is one of the connected components of $D$ we show that $D_0$ and $D\setminus D_0$
 are disconnected. This is equivalent to prove that if $s$ is the reflection in one of the
 planes that bound $D_0$, then $s\not\in f^{-1}(\Gamma_a)$. The definition of $D_0$
 implies $s=s_is^0s_i$ for some $s_i\in G^a$ and some reflection $s^0\in\{s_1,\dots,s_l\}\setminus f^{-1}(\Gamma_a)$ and therefore $s^0=s_iss_i\;\;\Rightarrow\;\;
 s^0\in G^a$ if $s\in f^{-1}(\Gamma_a)$.

For examples (i)-(iv) we have: $F=D_0=D$ in (i);\;$F\subsetneq D_0=D$ in (ii);\
;$F=D_0\subsetneq D=D_0\cup\sigma D_0$ in (iii);\;$F\subsetneq D_0=\mathrm{Int}\left(\overline{F}\cup g_1\overline{F}\right)\subsetneq D=D_0\cup\sigma D_0$ in (iv).  Figure 2 illustrates these properties.

We are now in a position to state the main results.
\begin{theorem}\label{theorem1}
Under Hypotheses \ref{h1}--\ref{h3}, and assuming that $G$ is a finite reflection group, there exists an $f$-equivariant classical solution $u$ to system \eqref{elliptic system}, and a positive decreasing map $q:[0,+\infty)\rightarrow\R$ with $\lim_{r\rightarrow+\infty}q(r)=0,$ such that
\begin{enumerate}
\item $u(\overline F) \subset \overline \Phi$ and $u(\overline{D})\subset\overline{\mathcal{D}}$,
\item $|u(x)-a| \leq q(d(x,\partial D)),\;\text{ for }\;x\in D$,
\item if the matrix $D^2W(a)$  is positive definite, then $q(r)\leq Ke^{-kr}$ for some constants $k, K>0$ and therefore
$|u(x)-a|\leq Ke^{-k d(x,\partial D)},\;\text{ for }\;x\in D$.
\end{enumerate}
\end{theorem}

\begin{theorem}\label{theorem2}
Under Hypotheses \ref{h1}--\ref{h3}, and assuming that $G$ is a discrete reflection group, there exists for every $R>R_0$  (a positive constant), a nontrivial $f$-equivariant classical solution $u_R$ to system
\begin{equation}\label{elliptic system R}
\Delta u_R - R^2 W_u(u_R) = 0,\,\,\,\hbox{for}\,\, x\in \R^n
\end{equation}
such that
\begin{enumerate}
\item $u_R(\overline {F}) \subset \overline \Phi$ and $u_R(\overline{D})\subset\overline{\mathcal{D}}$.
\item $|u_R(x)-a| \leq q(R d(x,\partial D)),\;\text{ for }\;x\in D$
where $q$ is as in Theorem \ref{theorem1},
\item $|u_R(x)-a| \leq K \mathrm{e}^{-kR d(x,\partial D)}$, for $x \in D$, \\for positive constants $k$, $K$, \medskip
provided $D^2W(a)$  is  positive definite.
\end{enumerate}
\end{theorem}

The solution $u_R$ of (\ref{elliptic system R}) given by Theorem \ref{theorem2} is periodic.
 We describe this periodic structure of $u_R$ under the assumption that the positive
 homomorphism $f$ satisfies \footnote{ For an example of a homomorphism that does not
 satisfy (\ref{f-mapsto-i}) take $\Gamma=\{I_m,\gamma\}$ with $\gamma$ the reflection
 in the plane $\{u_1=0\}$ and $G=\langle s_j\rangle_{j\in\Z}=T G_0$ where $s_j$ is the
 reflection in the plane $\{x_1=j\}$, $T$ is the translation group generated by the translation $t_0$ by the vector $(2,0,\ldots,0)$ and $G_0=\{I_n, s_0\}$. Define $f:G\rightarrow\Gamma$ by
\begin{equation*}
\begin{split}
& f(t_0)=f(s_{2 j})=\gamma,\\
& f(s_{2 j+1})=I_m.
\end{split}
\end{equation*}
 }
\begin{equation}\label{f-mapsto-i}
f(t)=I_m,\;\text{ for }\;t\in T.
\end{equation}
This covers the examples that we present below. On the other hand we are not aware of positive homomorphisms that do not satisfy (\ref{f-mapsto-i}). Assuming (\ref{f-mapsto-i}), if $G=T G_0$ with $G_0$ the point group of $G$ and $T$ its translation group we have
\[\begin{split}
 u_R(t x)=u_R(x),\;\text{ for }\;t\in T,\;x\in C\end{split}\]
 where
\[ C=\mathrm{Int}\cup_{g\in G_0}g\overline{F}\hskip2.7cm\]
is the elementary {\it cell}. $C$ is a convex polytope that satisfies
\[t C\cap C=\varnothing,\;\text{ for }\;t\in T\setminus\{I_n\}\hskip.8cm\]
and defines a tessellation of $\R^n$ as the union of the translations $t C,\,t\in T$ of $C$: $\R^n=\overline{\cup_{t\in T}t C}$. In this sense we can say that $u_R$ has a {\it crystalline} structure and that $C$ is the elementary {\it crystal}.

Let us illustrate Theorem \ref{theorem2} with the help of the example described in Sec.2, where the discrete reflection group $G'$ acts on the domain $x$-plane while the finite reflection group $\Gamma=D_3$ acts on the target $u$-plane.
We have already verified that the homomorphism $f':G' \to \Gamma$ is positive
and therefore Hypothesis \ref{h1} is satisfied. When Hypotheses \ref{h2}--\ref{h3} also hold, Theorem \ref{theorem2} ensures the existence,
for every $R$ sufficiently large, of a nontrivial $f'$-equivariant solution $u_R$ to \eqref{elliptic system R}, such that
$u_R(\overline {F'}) \subset \overline \Phi$. By $f'$-equivariance, the other fundamental
 domains of $G'$ are mapped into fundamental domains of $\Gamma$ as in Figure \ref{figure2}. Properties (ii) and (iii) state that for every $x \in D$, $u_R(x)$ approaches as
 $R$ grows, the unique minimum $a$ of $W$ in $\overline \Phi$, with a speed that
 depends on $d(x,\partial D)$.
If for instance the potential $W$ has six minima (one in the interior of each fundamental domain of $D_3$), then the set $D$ is the union of
the fundamental domains of $G'$ with the same colour (cf. Figure \ref{figure2} left) and
 $\mathcal{D}$ is the sector with the same color of $D$ (cf. Figure \ref{figure2} right). If
 $a$ lies on the boundary of the fundamental domain of $D_3$, for instance $a=(1,0)$ then $\mathcal{D}$ is
the $120$ degree sector that contains $a$ and $D$ is the union of all fundamental
 domains (triangles) with the same colors of the two sectors that compose $\mathcal{D}$. For this example condition (\ref{f-mapsto-i}) is satisfied and the elementary crystal $C$, is the
 hexagon determined by  the fundamental domains (triangles) whose closure contains $0\in\R^2$.

\begin{figure}[h]
\begin{center}
\includegraphics[scale=0.7]{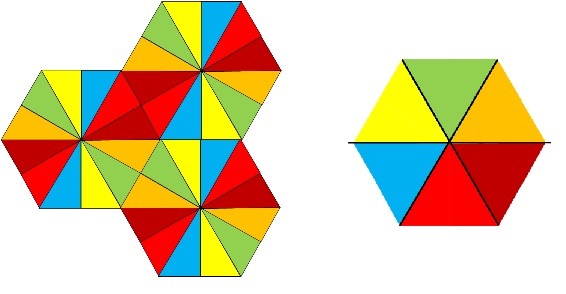}
\end{center}
\caption{Fundamental domains for the actions on $\R^2$ of $G^\prime$ (left) and $D_3$ (right).  The $f'$-equivariant solution $u_R$ of (\ref{elliptic system}) given by Theorem \ref{theorem2} maps fundamental domains into fundamental domains with the same color.}
\label{figure2}
\end{figure}

To give a first application of Theorem \ref{theorem1}, consider the  particular case where $G=D_6=\{I_2,\sigma \} \times D_3$, $\Gamma=D_3$, and $f$ the positive homomorphism  defined by $f(\gamma)=\gamma$, $f( \sigma \gamma)=\gamma$, for every $\gamma \in D_3$.
Figure \ref{figurenew} (iii) and (iv) shows the correspondence of the fundamental domains by $f$- equivariant solutions $u$ of (\ref{elliptic system}) when the potential $W$ has respectively six and three minima. The case of $W$ with three minima when $u$ has a six-fold structure, see Figure \ref{figurenew} (iv) was first considered in \cite{bates-fusco}. Other examples of application of Theorem \ref{theorem1} and Theorem \ref{theorem2} are given in Sec.5 and Sec.6.

\section{Proofs of Theorem \ref{theorem1} and Theorem \ref{theorem2}}
\begin{step}[Minimization] \end{step}
In the case where $G$ is a finite reflection group, we first construct the solution in a ball $B_R \subset \R^n$ of radius $R>0$ centered at
the origin. We set $F_R=F \cap B_R$ and define the class
$$A^R := \left\{ u \in W^{1,2}(B_R,\R^m) : u \text{ is } f \text{-equivariant and } u(\overline{F}_R) \subset \overline{\Phi} \right\}$$ in which we have imposed the positivity constraint $ u(\overline{F}_R) \subset \overline{\Phi}$. Then, we consider the minimization problem
\[
\min_{A^R} J_{B_R}, \text{ where } J_{B_R}(u) = \int_{B_R} \Big{(} \frac{1}{2} |\nabla u|^2 + W(u) \Big{)} \dd x.
\]
Since $A^R$ is convex and hence weakly closed in $W^{1,2}(B_R,\R^m)$, a minimizer $u_R$ exists, and because of Hypothesis \ref{h2} we can assume that
\begin{eqnarray}\label{below-emme}
|u_R(x)| \leq M.
\end{eqnarray}

In the case where $G$ is a discrete reflection group, we can work directly in the fundamental domain $F$. Suppose first that $F$ is bounded. Then, we consider the class
$$A:=\{u \in W^{1,2}_{\mathrm{loc}}(\R^n,\R^m) :  u \text{ is $f$-equivariant and such that $u(\overline{F}) \subset \overline{\Phi}$} \},$$
and choose a minimizer $u_R$ of the problem
\[
\min_{A} J_{F}^R, \text{ where } J_{F}^R (u) = \int_{F} \Big{(} \frac{1}{2} |\nabla u|^2 + R^2 W(u) \Big{)} \dd x,
\]
satisfying the estimate (\ref{below-emme}).

Now, suppose that $F$ is not bounded. This implies that all the reflection hyperplanes of $G$ are parallel to a subspace
$\{ 0 \}^{\nu} \times \R^d \subset \R^n$ (with $\nu + d=n$, $d \geq 1$), and that $G$ also acts in $\R^{\nu}$. Since $F=F_{\nu} \times \R^d$,
 with $F_{\nu} \subset \R^{\nu}$ bounded, we have, according to the preceding argument, a minimizer $v_R: \R^{\nu} \to \R^m$ of $$J_{F_{\nu}}^R (v) = \int_{F_{\nu}} \Big{(} \frac{1}{2} |\nabla_{x_{\nu}} v|^2 + R^2 W(v) \Big{)} \dd x_{\nu} ,$$
 in the analog of $A$ with $n$ replaced by $\nu$, that is, the class of $W^{1,2}_{\mathrm{loc}}(\R^\nu,\R^m)$ maps $v$, which are
$f$-equivariant and satisfy $v(\overline{F_{\nu}} ) \subset \overline \Phi$. Then, we set $u_R(x):=v_R(x_{\nu})$, where $x=(x_{\nu},x_d) \in \R^n$.

\begin{step}[Removing the positivity constraint with the gradient flow] \end{step}
To show that the positivity constraint built in $A^R$ (or $A$) does not affect the Euler-Lagrange equation we will utilize the gradient flow associated to the elliptic system.
In the case where $G$ is a finite reflection group we consider
\begin{equation}\label{evolution-problem1}
\begin{cases}
\dfrac{\partial u}{\partial t} = \Delta u - W_u(u), &\text{in } B_R \times (0,\infty), \bigskip \\
\dfrac{\partial u}{\partial \n} = 0, &\text{on } \partial B_R \times (0,\infty), \bigskip \\
u(x,0) = u_0(x), &\text{in } B_R,
\end{cases}
\end{equation}
where ${\partial}/{\partial \n}$ is the normal derivative.

In the case where $G$ is a discrete reflection group, we consider
\begin{equation}\label{evolution-problem2}
\begin{cases}
\dfrac{\partial u}{\partial t} = \Delta u - R^2 W_u(u), &\text{in } \R^n \times (0,\infty), \bigskip \\
u(x,0) = u_0(x), &\text{in } \R^n.
\end{cases}
\end{equation}
Since $W$ is $C^3$, the results in
\cite{henry} apply and provide a
unique solution to \eqref{evolution-problem1} (or \eqref{evolution-problem2}) which is
 smooth if we assume that $u_0$ is globally Lipschitz. In the next two lemmas we will
 establish that the gradient flow preserves the $f$-equivariance and the positivity of a smooth initial condition.

\begin{lemma}\label{equivariance flow}
Under Hypothesis \ref{h2}, if the initial condition $u_0$ is a smooth, $f$-equivariant map, then for every $t>0$, the solution $u(\cdot,t)$ of problem \eqref{evolution-problem1} (or \eqref{evolution-problem2}) is also $f$-equivariant.
\end{lemma}
\begin{proof} We only write the proof for \eqref{evolution-problem1}, since it is identical for \eqref{evolution-problem2}.
Let $g \in G$ and $\gamma:=f(g)\in \Gamma  < O(\R^m)$. We are going to check that for
 every $x \in B_R$ and every $t>0$, we have $u(gx,t)=\gamma u(x,t)$ or equivalently
 $u(x,t)=\gamma^{\top} u(gx,t)$.  Let us set $v(x,t):=\gamma^{\top} u(gx,t)$.  Since $g$ is
 an isometry, $\Delta v(x,t)=\gamma^{\top} (\Delta u)(gx,t)$. On the other hand, we have
 $\frac{\partial v}{\partial t}(x,t)= \gamma^{\top} \frac{\partial u}{\partial t}(gx,t)$, and
 utilizing the symmetry of the potential: $W_u (v(x,t))=\gamma^{\top} W_u(u(gx,t))$.
 Finally, we see that for $x \in \partial B_R$ and $t>0$, $\dfrac{\partial v}{\partial \n}(x,t)=\gamma^{\top}\dfrac{\partial u}{\partial  g \n }(gx,t)=0$. Thus, $v$ is also a smooth
 solution of  \eqref{evolution-problem1} with initial condition $v_0(x)=\gamma^{\top} u_0(gx)=u_0(x)$, and by uniqueness $u\equiv v$.
\end{proof}

\begin{lemma}\label{positivity flow}
Under Hypotheses \ref{h1}--\ref{h2}, and assuming that the initial condition $u_0$ is a smooth map, we have:
\begin{itemize}
\item $u_0 \in A^R \Rightarrow u(\cdot,t) \in A^R, \ \forall t>0$, when $G$ is a finite reflection group.
\item $u_0 \in A \Rightarrow u(\cdot,t) \in A, \ \forall t>0$, when $G$ is a discrete reflection group.
\end{itemize}
\end{lemma}
\begin{proof}
We first present the proof in a specific case where the argument can be described with
 simpler notation  and then consider the abstract situation and give the proof for the general case. The case we discuss first is the example of Sec.2 where we have the discrete
 reflection group $G'$ acting on the domain, $\Gamma=D_3$ acting on the target, and the
 homomorphism $f':G' \to \Gamma$. We conserve the notation of Sec.2 and refer to Figure
 \ref{figure1} and to the comments following Theorem \ref{theorem2}. In particular, we still
 denote by $F'$ the fundamental domain of $G'$ and by $f'$ the homomorphism $G' \to
 \Gamma$ (which are denoted by $F$ and $f$ in the statement of Theorem \ref{theorem2} and in the first step of its proof).
We also denote by $\rho_1$ and $\rho_2$ the outward unit normals to the lines $\Pi_1$ and $\Pi_2$ that bound the fundamental domain $\Phi$ of $\Gamma$. Let
 $\Pi_3=\{u_2=\sqrt{3} u_1 \}$ be the third reflection line of $\Gamma$ and let
 $\rho_3:=(-\sqrt{3}/2,1/2) \bot \Pi_3$. From (\ref{evolution-problem2}) and the symmetry of $W$ given by (\ref{g-invariance}),
it follows
that for every $j=1,2,3$, the projection $h_j(x,t):= \langle u(x,t), \rho_j \rangle$ satisfies the linear scalar equation:
$$\Delta h_j +c_j^* h_j-\dfrac{\partial h_j}{\partial t} =0 \text{ in } \R^2 \times (0,+\infty),$$
with $c_j^*=R^2 c_j$ and $c_j$ (cf. \eqref{coeff cj} below) continuous and bounded on $\R^2 \times [0,T]$, for every $T>0$ .

Now, suppose that for some $t_0>0$, $u(.,t_0)$ does not belong to the class $A$.

In order to have equations with  nonpositive coefficients, we do the standard transformation and set $\tilde{h}_j(x,t):=e^{-\lambda t}h_j(x,t)$, where the constant
 $\lambda$ is chosen larger than $\sup \{c_j^*(x,t) \mid x \in \R^2, \  t\in [0,t_0],  \ j=1,2,3 \}$. Then, we have
\begin{equation}\label{linear equation j}
\Delta \tilde{h}_j + \tilde{c}_j^*  \tilde{h}_j-\dfrac{\partial
 \tilde{h}_j}{\partial t} =0 \text{ in } \R^2 \times (0,t_0], \text{  with }  \tilde{c}_j^*=c_j^*-  \lambda    \leq 0.
\end{equation}
Let $\mu := \max \{d (  e^
{ -\lambda  t} u(x,t), \overline \Phi) \mid x\in \overline{F'}, \  t\in [0,t_0] \}>0$, and suppose that this is achieved at $\tilde x \in \overline{F'}$ at time $\tilde t \in (0,t_0]$ (since $u_0 \in A$).
Define $$ \tilde u:=e^
{ -\lambda \tilde t}u(\tilde x, \tilde t),  \  \    \rho := \frac{ \tilde u- \tilde v}{ | \tilde  u - \tilde v|},$$ where $\tilde v$ is the unique point of $\partial \Phi$  (since $\Phi$ is convex) such that $d(\tilde u , \tilde v)=\mu$.
According to the direction of $\rho$, we distinguish the following cases:
 \begin{itemize}
\item[(i)] If $\rho=\rho_1$, then $\tilde v  \in \Pi_1 \cap \overline \Phi$ and we define
 $\omega := \{ (x,t) \in \R^2 \times (0,t_0] :    \langle e^{-\lambda t} u(x,t), \rho_1 \rangle
 > 0\}$. Clearly, $(\tilde x, \tilde t) \in \omega$ which is relatively open in $\R^2 \times (0,t_0]$.
\item[(ii)] If $\rho=\rho_2$, then $\tilde v  \in \Pi_2 \cap \overline \Phi$. Similarly, define $\omega := \{ (x,t) \in \R^2 \times (0,t_0] :    \langle e^{-\lambda t} u(x,t), \rho_2 \rangle > 0\}$, and we have $(\tilde x, \tilde t) \in \omega$ which is relatively open in $\R^2 \times (0,t_0]$.
\item[(iii)] If $\rho=\rho_3$, then $\tilde v=0$. We check that $(\tilde x, \tilde t) \in \omega := \{ (x,t) \in \R^2 \times (0,t_0] :    \langle e^{-\lambda t} u(x,t), \rho_3 \rangle > 0\},$ which is relatively open in $\R^2 \times (0,t_0]$.
\item[(iv)] If $\rho=\alpha_1 \rho_1 +\alpha_3 \rho_3$ with $\alpha_1, \alpha_3>0$, then
 $\tilde v=0$ and we define \\ $\omega := \{ (x,t) \in \R^2 \times (0,t_0] :    \langle
 e^{-\lambda t} u(x,t), \rho_j \rangle > 0 \text{ for $j=1$ and $j=3$} \}$. Thanks to the fact that $\langle \rho_1, \rho_3 \rangle \geq 0$, we have again $(\tilde x, \tilde t) \in \omega$ which is relatively open in $\R^2 \times (0,t_0]$.
\item[(v)] If, $\rho=\alpha_2 \rho_2 +\alpha_3 \rho_3$ with $\alpha_2, \alpha_3>0$, then $\tilde v=0$ and we define $\omega$ in a similar way.
\end{itemize}
We want to apply the maximum principle to $\tilde{h}(x,t):= \langle e^
{ -\lambda  t}  u(x,t), \rho \rangle$ in a neighborhood of $(\tilde x ,\tilde t)$. In the cases (i), (ii) and (iii) above, the equation \eqref{linear equation j} trivially holds in $\omega$. In the cases (iv) and (v), we have the inequality $$\Delta \tilde{h} + \tilde{c}^*  \tilde{h}-\dfrac{\partial
 \tilde{h}}{\partial t} \geq 0 \text{ in } \omega , \text{  with }  \tilde{c}^*=\max \{ \tilde{c}_j^* \mid j=1,2,3 \} \leq 0.$$
Indeed, we can check that for instance in the case (iv):
\begin{equation}
\Delta \tilde{h} + \tilde{c}^* \tilde{h}-\dfrac{\partial
 \tilde{h}}{\partial t}=\alpha_1   (\tilde{c}^*-\tilde{c}^*_1) \tilde{h}_1+\alpha_3   (\tilde{c}^*-\tilde{c}^*_3) \tilde{h}_3 \geq 0, \ \forall (x,t) \in \omega. \nonumber
\end{equation}

At this stage, the fact that $\Phi$ is an acute angle sector (i.e. $\langle \rho_1, \rho_2 \rangle \leq 0$) is essential to conclude the proof. This property implies that
\begin{equation}\label{acute angles bis}
\tilde u \in \Pi_j \text{ (with $j=1,2$)} \Rightarrow \tilde v, \ \rho \in \Pi_j,
\end{equation}
and as a consequence
\begin{equation}\label{localization bis}
f'(s_i) \notin \Gamma_{\rho} \Rightarrow \tilde x \notin P_i,
\end{equation}
where $P_i$ ($i=1,2,3$) is a line bounding $F'$, corresponding to the reflection $s_i \in G'$. To show \eqref{localization bis}, suppose that $\tilde x \in P_i$. Then, $s_i(\tilde x)=\tilde x$, and by $f'$-equivariance
$$u(\tilde x,\tilde t)=u(s_i(\tilde x),\tilde t)=f'(s_i) u(\tilde x,\tilde t) \Leftrightarrow u(\tilde x, \tilde t) \in \ker (f'(s_i) -I_2) \text {}.$$
Since $\ker (f'(s_i) -I_2)$ is either the line $\Pi_1$ or the line $\Pi_2$ (cf. Hypothesis \ref{h1}), we deduce thanks to \eqref{acute angles bis} that
$$\rho \in \ker (f'(s_i) -I_2) \Leftrightarrow f'(s_i) \in \Gamma_{\rho}.$$

Property \eqref{localization bis} will enable us to locate $\tilde x$ in $\overline{ F'}$. Let $\tilde G'$ be the subgroup of $G'$ generated by the reflections $s_i$ ($i=1,2,3$) such that $f'(s_i) \in \Gamma _{\rho}$. By $f'$-equivariance we have
$$\mu=\tilde h(\tilde x, \tilde t)= \max \{ \tilde h(x,t) : x\in {\cup_{g\in \tilde G'} g\overline{F'}}, \  t\in [0,t_0] \}.$$

But now $\tilde x \in  \mathrm{Int}\left( {\cup_{g\in \tilde G'} g\overline{F'}} \right)$, thus thanks to the maximum principle for parabolic equations applied in $\omega$, we can see that $\tilde h(x,\tilde t) \equiv \mu$, for $x \in B_{\delta}(\tilde x) \cap \overline{F'}$ (where $\delta>0$). To finish the proof, we are going to show that the set $S:=\{y \in \overline{F'} : \tilde h(y,\tilde t)=\mu \}$ is relatively open. Indeed, let $y \in S$ and let $w$ be the projection of $e^
{ -\lambda \tilde t}u(y, \tilde t)$ on $\overline \Phi$. We have $e^{ -\lambda \tilde t}u(y, \tilde t)-w=\mu \rho$, and repeating the above argument we find $\tilde h(x,\tilde t) \equiv \mu$, for $x \in B_{\delta'}(y) \cap \overline{F'}$ (where $\delta'>0$). Thus, by connectedness $\tilde h(.,\tilde t) \equiv \mu>0$ on $\overline{ F'}$ which is a contradiction since $\tilde h (0, \tilde t)=0$.

Let us now give the proof of the Lemma for arbitrary groups. We just present it when $G$ is
 a finite reflection group since it is similar for discrete reflection groups. We will need to apply the maximum principle to some projections of the solution $u$. We denote by $\rho_1,\ldots, \rho_k$ the outward unit
 normals to the hyperplanes $\Pi_1,\ldots, \Pi_k$ that bound the fundamental domain $\Phi$ (see Definition
 \ref{positive homomorphism}). We also consider the collection $\Pi_1,\ldots,\Pi_q$ ($k \leq q$) of all the reflection hyperplanes of $\Gamma$, and choose for $j>k$, a unit normal $\rho_j$ to $\Pi_j$.
Since the potential $W$ is symmetric, we know
that for every $j=1,\ldots,q$, the projection $h_j(x,t):= \langle u(x,t), \rho_j \rangle$ satisfies the linear scalar equation:
$$\Delta h_j +c_j h_j-\dfrac{\partial h_j}{\partial t} =0 \text{ in } B_R \times (0,+\infty),$$ with
\begin{equation}\label{coeff cj}
c_j =-\left\langle
\int_0^1W_{uu}\big(u + (s-1)\langle u, \rho_j \rangle
\rho_j \big) \rho_j \,\dd s,\, \rho_j \right\rangle.
\end{equation}
From this formula, one can see that $c_j$ is continuous and bounded on $B_R \times [0,T]$, for every $T>0$.

Now, suppose that for some $t_0>0$, $u(.,t_0)$ does not belong to the class $A^R$.

In order to have an equation with a nonpositive coefficient, we again do the standard transformation and set $\tilde{h}_j(x,t):=e^{-\lambda t}h_j(x,t)$, where the constant
 $\lambda$ is chosen bigger than $\sup \{c_j(x,t) : x \in B_R, \  t\in [0,t_0],  \ j=1,\ldots,q \}$. Then, we have $$\Delta \tilde{h}_j + \tilde{c}_j  \tilde{h}_j-\dfrac{\partial
 \tilde{h}_j}{\partial t} =0 \text{ in } B_R \times (0,t_0], \text{  with }  \tilde{c}_j=c_j-  \lambda    \leq 0.$$

Let $\mu := \max \{d (  e^
{ -\lambda  t} u(x,t), \overline \Phi) \mid x\in \overline{F_R}, \  t\in [0,t_0] \}>0$, and suppose that this is achieved at $\tilde x \in \overline{F_R}$ at time $\tilde t \in (0,t_0]$ (since $u_0 \in A_R$).
Define $$ \tilde u:=e^
{ -\lambda \tilde t}u(\tilde x, \tilde t),  \  \    \rho := \frac{ \tilde u- \tilde v}{ | \tilde  u - \tilde v|},$$ where $\tilde v$ is the unique point of $\partial \Phi$  (since $\Phi$ is convex)
 such that $d(\tilde u , \tilde v)=\mu$. We will apply the maximum principle to $\tilde{h}(x,t):= \langle e^
{ -\lambda  t}  u(x,t), \rho \rangle$ in a neighborhood of $(\tilde x ,\tilde t)$ in $\overline{B_R} \times (0,t_0]$. To do this, in analogy to what was done in the special case considered above we need to consider various possibilities for the unit vector $\rho$. If $\tilde v$ belongs to the interior of a $m-p$ dimensional face $\Pi_1 \cap \ldots \cap \Pi_p \cap \overline \Phi$ ($1 \leq p \leq k$) of $\Phi$,
 then, using also that $\rho_1,\ldots, \rho_k$ are
 linearly independent, we have 
$$\rho \bot E, \ E:= \Pi_1 \cap \ldots \cap \Pi_p, \text{ that is, } \rho \in E^{\bot}=\R \rho_1 \oplus \ldots \oplus \R \rho_p$$
 where $\R \rho_j=\{x: x=t\rho_j, t\in\R\}$ and $E^{\bot}$ is the orthogonal complement of $E$. Let $\tilde \Gamma$ be the subgroup of $\Gamma$ generated by the reflections with respect to the hyperplanes $\Pi_1, \ldots, \Pi_p$. The elements of
 $\tilde \Gamma$ leave invariant the subspace $E$, and actually $\tilde \Gamma $ acts in
 $E^{\bot}$. Let $ N \supset \{\pm \rho_1,\ldots,\pm \rho_p \}$, be the set of all the unit normals to the reflection hyperplanes of $\tilde\Gamma$.  We claim that
\begin{equation}\label{expression rho}
\left\{\begin{array}{l}
\rho=\alpha_1 \nu_1+ \ldots+\alpha_{\tilde p} \nu_{\tilde p},\;\text{ with }\;1 \leq \tilde p \leq p, \ \alpha_1, \ldots, \alpha_{\tilde p}>0,\; \nu_1,\ldots,\nu_{\tilde p}\in N,\\\\
\;\text{ and }\;\langle \nu_i,\nu_j\rangle \geq 0,\;\text{ for }\;1 \leq i,j\leq \tilde p.
\end{array}\right.
\end{equation}
Given $p$ linearly independent vectors $\nu_1,\ldots,\nu_p\in N$ we denote by $C(\nu_1,\ldots,\nu_p)$ the cone
\begin{equation}\label{cone-def}
C(\nu_1,\ldots,\nu_p):=\{ \alpha_1 \nu_1+  \ldots +\alpha_p \nu_p \mid \alpha_1 ,\ldots, \alpha_p \geq 0 \}.
\end{equation}
 To prove the claim we start by observing that (since $\rho_1,\dots,\rho_p$ are linearly independent) we have $\rho\in\mathcal{C}$ with
\[ \mathcal{C}=C(\rho_1^\prime,\dots,\rho_p^\prime),\;\text{ with }\;\rho_j^\prime=\rho_j\,\text{ or }-\rho_j,\;j=1,\ldots,p.\]
To conclude the proof we show that $\mathcal{C}$ can be partitioned in cones of the form (\ref{cone-def}) that satisfy the condition
\begin{equation}\label{theonly-nu}
C(\nu_1,\ldots,\nu_p)\cap N=\{\nu_1,\ldots,\nu_p\}.
\end{equation}
Note that (\ref{theonly-nu}) and the fact that $\tilde{\Gamma}$ is a reflection group automatically imply
\[\langle\nu_i,\nu_j\rangle\geq 0,\;\text{ for }\;1\leq i,j\leq  p.\]
Indeed $N$ is invariant under $\tilde{\Gamma}$ ($N$ is the root system of $\tilde{\Gamma}$) and therefore $\langle\nu_i,\nu_j\rangle<0$ implies that the reflection $\nu\not\in\{\nu_i,\nu_j\}$ of $\nu_j$ in the hyperplane orthogonal to $\nu_i$ belongs to both $N$ and $C(\nu_1,\ldots,\nu_p)$ in contradiction with (\ref{theonly-nu}).
If $\mathcal{C}$ does not satisfy (\ref{theonly-nu}) there exists $\nu\in N$ that (possibly after a renumbering of the vectors $\rho_1^\prime,\dots,\rho_p^\prime$) has the form $\nu=\alpha_1 \rho_1^\prime+ \ldots+\alpha_{\hat p} \rho_{\hat p}^\prime$
with $2 \leq \hat p \leq p$, $\alpha_1, \ldots, \alpha_{\hat p}>0$ and we can partition $\mathcal{C}$ into the $\hat{p}$ cones $\mathcal{C}_i=
C(\rho_1^\prime,\ldots,\rho_{i-1}^\prime,\nu,\rho_{i+1}^\prime,\ldots,\rho_p^\prime)$, $i=1,\ldots,\hat{p}$ defined by the linearly independent vectors $\rho_1^\prime,\ldots,\rho_{i-1}^\prime,\nu,\rho_{i+1}^\prime,\ldots,\rho_p^\prime$, $i=1,\ldots,\hat{p}$. If $\mathcal{C}_i$ does not satisfy (\ref{theonly-nu}) we partition $\mathcal{C}_i$ in the same fashion used for $\mathcal{C}$ and continue in this way. Note that at each step (if some of the cones of the partition does not satisfies (\ref{theonly-nu})) the number of vectors in $N$ used to generate the cones of the partition increases by one. Therefore, since $N$ is a finite set, the process terminates after a finite number of steps exactly when all the cones of the partition satisfy (\ref{theonly-nu}). This conclude the proof of the claim.

Since $\tilde v \in E$, it follows that (with $\nu_1\ldots,\nu_{\tilde{p}}$ the vectors in \eqref{expression rho})
$$(\tilde x, \tilde t) \in \omega := \{ (x,t) \in \overline{B_R} \times (0,t_0] \mid   \  \langle e^{-\lambda t} u(x,t), \nu_j \rangle > 0, \  \forall j=1,\ldots,\tilde  p \},$$ which is relatively open in $\overline{B_R} \times (0,t_0]$, and in addition we have $$\Delta \tilde{h} + \tilde{c}  \tilde{h}-\dfrac{\partial
 \tilde{h}}{\partial t} \geq 0 \text{ in } \omega , \text{  with }  \tilde{c}=\max \{ \tilde{c_j} \mid j=1,\ldots,q \} \leq 0.$$
 At this stage, the fact that $\Phi$ has acute angles (i. e. $\langle \rho_i, \rho_j \rangle \leq 0$ for $1 \leq i< j \leq k$) is essential to conclude the proof. This property implies that
\begin{equation}\label{acute angles}
\tilde u \in \Pi_j \text{ (with $1 \leq j \leq k$)} \Rightarrow \tilde v, \ \rho \in \Pi_j,
\end{equation}
and as a consequence
\begin{equation}\label{localization}
f(s_i) \notin \Gamma_{\rho} \Rightarrow \tilde x \notin P_i,
\end{equation}
where $P_i$ is, as in Definition \ref{positive homomorphism}, a hyperplane bounding $F$ corresponding to the reflection $s_i \in G$. To show \eqref{localization}, suppose that $\tilde x \in P_i$. Then  $$u(\tilde x, \tilde t) \in \ker (f(s_i) -I_m) \text { by $f$-equivariance},$$
and thanks to Hypothesis \ref{h1} and \eqref{acute angles},
$$\rho \in \ker (f(s_i) -I_m) \Rightarrow f(s_i) \in \Gamma_{\rho}.$$

Property \eqref{localization} will enable us to locate $\tilde x$ in $\overline F_R$. Let $\tilde G$ be the subgroup of $G$ generated by the reflections $s_i$ ($i=1,\ldots,l$) such that $f(s_i) \in \Gamma _{\rho}$. By $f$-equivariance we have
$$\mu=\tilde h(\tilde x, \tilde t)= \max \{ \tilde h(x,t) \mid x\in {\cup_{g\in \tilde G} g\overline{F_R}}, \  t\in [0,t_0] \}.$$
But now, either $\tilde x \in  \mathrm{Int}\left( {\cup_{g\in \tilde G} g\overline{F_R}} \right)$ or  $\tilde x \in  \mathrm{Int}\left( {\cup_{g\in \tilde G} g\overline{F}} \right) \cap \partial B_R$.
In both cases, we can see, thanks to the maximum principle for parabolic equations applied in $\omega$ and thanks to Hopf's Lemma,
that $\tilde h(x,\tilde t) \equiv \mu$, for $x \in B_{\delta}(\tilde x) \cap \overline{F_R}$ (where $\delta>0$). To finish the proof, we are
going to show that the set $S:=\{y \in \overline{F_R} \mid \tilde h(y,\tilde t)=\mu \}$ is relatively open. Indeed, let $y \in S$ and let $w$
be the projection of $e^{ -\lambda \tilde t}u(y, \tilde t)$ on $\overline \Phi$. We have $e^{ -\lambda \tilde t}u(y, \tilde t)-w=\mu \rho$
and according to the preceding argument $\tilde h(x,\tilde t) \equiv \mu$ for $x \in B_{\delta'} (y) \cap \overline{F_R}$ (where $\delta'>0$). Thus, by connectedness $\tilde h(.,\tilde t) \equiv \mu>0$ on $\overline F_R$ which is a contradiction since $\tilde h (0, \tilde t)=0$.
\end{proof}

If the initial condition in \eqref{evolution-problem1} (respectively \eqref{evolution-problem2}) is a $W^{1,2}(B_R, \R^m)$ (respectively,
 $W^{1,2}_{\mathrm{loc}}(\R^n,\R^m)$), bounded map, then the solution to \eqref{evolution-problem1} (respectively,
\eqref{evolution-problem2}) belongs to $C([0,\infty), W^{1,2}(B_R,\R^m))$ (respectively $C([0,\infty), W^{1,2}_{\mathrm{loc}}(\R^n,\R^m))$),
and is smooth for $t>0$. We are now going to take the minimizer $u_R$ constructed in Step 1, as the initial condition in
\eqref{evolution-problem1} (respectively \eqref{evolution-problem2}). Thanks to Lemma \ref{approximation} below (and to its analog for
discrete reflection groups), we can construct a sequence of smooth, $f$-equivariant, and positive maps
$(u_k)$ which converges to $u_R$ for the $W^{1,2}$ norm, as $k \to \infty$.
Applying then Lemmas \ref{equivariance flow} and \ref{positivity flow} to $u_k$, and utilizing the continuous dependence for the flow of the initial
condition, we obtain that the solution to \eqref{evolution-problem1} (respectively \eqref{evolution-problem2}), with initial condition $u_R$, is $f$-equivariant and positive, that is, $u(\cdot, t; u_R) \in A^R, \text{ for }
t \geq 0$ (respectively $u(\cdot, t; u_R) \in A, \text{ for }
t \geq 0$).

\begin{lemma}\label{approximation}
Let $u \in W^{1,2}(B_R,\R^m) \cap L^{\infty}(B_R,\R^m)$ be a $f$-equivariant map such that $u ( \overline{F_R} ) \subset \overline{\Phi}$. Then, there exists a sequence $(u_k) \subset
C(\overline{B_R},\R^m)$ of globally Lipschitz maps with the following
 properties:
\begin{itemize}
\item[(i)] $ u_k$ is $f$-equivariant,
\item[(ii)] $\| u_k \|_{L^{\infty} (B_R,\R^m)} \leq \| u \|_{L^{\infty} (B_R,\R^m)}$,
\item[(iii)] $u_k( \overline{F_R} ) \subset \overline{\Phi}$ (positivity),
\item[(iv)] $u_k$ converges to $u$ in $W^{1,2}(B_R,\R^m)$,
  as $k \to \infty$.
\end{itemize}
\end{lemma}
\begin{proof}
See Proposition 5.2 in \cite{alikakos-smyrnelis}.
\end{proof}
Thanks to the fact that $u_R$ is a
global minimizer of $J_{B_R}$ (respectively $J_{F}^R$) in $A^R$ (respectively $A$), and since
$u(\cdot,t;u_R)$
is a classical solution to \eqref{evolution-problem1} (respectively \eqref{evolution-problem2}) for $t>0$, we conclude from the computation
\begin{equation}\label{decreasing J}
\frac{\mathrm{d}}{\mathrm{d} t} J_{B_R} (u(\cdot,t))= - \int_{B_R} |u_t|^2 \,\mathrm{d} x \ \left( \text{respectively } \frac{\mathrm{d}}{\mathrm{d} t} J^R_{F} (u(\cdot,t))
= - \int_{F} |u_t|^2 \,\mathrm{d} x \right)
\end{equation}
that $|u_t (x,t)|=0$, for all $x \in B_R$ (respectively $x \in F$) and $t>0$.
Hence, for $t>0$, $u(\cdot, t)$ satisfies
\begin{equation}\label{u-satisfies}
\Delta u(x,t) - W_u (u(x,t)) = 0 \ \text{ (respectively } \Delta u(x,t) - R^2 W_u (u(x,t)) = 0).
\end{equation}
By taking $t \to 0+$ and utilizing the continuity of the flow in
$W^{1,2} (B_R,\R^m)$ (respectively $W^{1,2}_{\mathrm{loc}}(\R^n,\R^m)$) at $t=0$,
we obtain that $u_R$ is a $f$-equivariant, classical
solution to system \eqref{elliptic system} (respectively \eqref{elliptic system R}) satisfying also
$u_R( \overline{ F_R}) \subset \overline{\Phi}$ (respectively $u_R( \overline{ F}) \subset \overline{\Phi}$).

If $G$ is a finite reflection group, from the family of solutions $u_R\in C^3(B_R,\R^m),\;R\geq 1$  we can deduce the existence of an entire $f$-equivariant classical solution $u\in C^3(\R^n,\R^m)$ to system \eqref{elliptic system} defined by
\begin{eqnarray}\label{entire}
u(x)=\lim_{j\rightarrow+\infty}u_{R_j}(x),
\end{eqnarray}
where $R_j\rightarrow+\infty$ is a suitable subsequence and the convergence is in the $C^2$ sense in compact subsets of $\R^n$. This follows from the fact that $u_R$ satisfies the bound
\begin{eqnarray}\label{below-emmeprime}
\|u_R\|_{C^{2+\alpha}(\overline{B_R},\R^m)}\leq M^\prime,
\end{eqnarray}
for some $\alpha\in(0,1)$ and $M^\prime>0$ independent of $R\geq 1$. The estimate (\ref{below-emmeprime}) follows by elliptic regularity  from \eqref{below-emme}, from the assumed smoothness of $W$ and from the fact that $\partial B_R$ is uniformly smooth for $R\geq 1$. The solution $u$ satisfies also: $u(\overline{F})\subset\overline{\Phi}$.
\begin{step}[Pointwise estimates] \end{step}
From Step.2 to complete the proof of Theorem \ref{theorem1} and Theorem \ref{theorem2} it remains to prove that the entire solution $u$ to system \eqref{elliptic system}  and the
 solution $u_R$ to system  \eqref{elliptic system R} satisfy the pointwise estimates stated in
 Theorem \ref{theorem1} and Theorem \ref{theorem2} respectively. For establishing these estimates we utilize the following theorem that we quote from \cite{f}. We need the concept of {\it local minimizer}
\begin{definition}\label{definition}
Let $\Omega\subset\R^n$ an open set.
A map $u\in C^2(\Omega;\R^m)\cap L^\infty(\Omega,\R^m)$ is a local minimizer if
\begin{eqnarray}
J_{\Omega^\prime}(u)\leq J_{\Omega^\prime}(u+v),
\end{eqnarray}
for all $v\in W_0^{1,2}(\Omega^\prime,\R^m)\cap L^\infty(\Omega^\prime,\R^m)$
and for all smooth bounded open subset $\Omega^\prime\subset\Omega$.
\end{definition}
\begin{theorem}\label{teo-2}
Assume $W:\R^m\to\R$ is a nonnegative $C^3$  potential and that $W(a)=0$ for some $a\in\R^m$ satisfying condition (ii) in Hypothesis \ref{h2}. Let $Z\subset\R^m$ be the set $Z=\{z\neq a: W(z)=0\}$ and let $Z_\delta=\cup_{z\in Z}B_{z,\delta}$ a $\delta$-neighborhood of $Z$.

 Assume that $u:\Omega\to\R^m$ is a local minimizer and that there is $M>0$ such that
 \begin{eqnarray}\label{new-emmeprime}
 \vert u\vert+\vert\nabla u\vert\leq M,\;\;x\in\Omega
 \end{eqnarray}

 Then there is a $\bar{q}\in(0,q^*]$ and a strictly decreasing continuous function $r:(0,\bar{q}]\to (0,\infty)$, $\lim_{q\rightarrow 0^+}r(q)=+\infty$, such that the condition
\begin{eqnarray}\label{condition}
B_{r(q)}(x_0)\subset\Omega\;\,\text{ and }\;\,u(B_{r(q)}(x_0))\cap Z_\delta=\varnothing,
\end{eqnarray}
implies
\begin{eqnarray*}
\vert u(x_0)-a\vert< q.
\end{eqnarray*}
The map $r$ depends only on $W$ and $M$  if $Z=\varnothing$ and also on $\delta$ otherwise.

Moreover, if $a$ is nondegenerate in the sense that  $D^2W(a)$ is positive definite, then there exists a constant $k_0>0$ such that
\begin{eqnarray}\label{log}
r(q)\leq r(\bar{q})+\frac{1}{k_0}\vert\log{\frac{q}{\bar{q}}}\vert,\;\;q\in(0,\bar{q}].
\end{eqnarray}
\end{theorem}

\begin{remark}\label{remark}
The map $r$ has a strictly decreasing  inverse $q:[r(\bar{q}),+\infty)\rightarrow (0,\bar{q}]$.
If $a$ is nondegenerate (\ref{log}) implies
\begin{eqnarray}\label{inverse}
q(r)<K_0e^{-k_0r},\;r\in[r(\bar{q}),+\infty),
\end{eqnarray}
for some $K_0>0$.
\end{remark}
For the proof of Theorem \ref{teo-2} we refer to \cite{f}, \cite{fusco}.  In \cite{f} the proof is for the case of a generic potential and covers all cases where $D_0=F$. To treat the general case  $F\subset D_0$ one needs to show that Theorem \ref{teo-2} holds true in the case of
 symmetric potentials and $f$-equivariant local minimizers. In \cite{fusco} the validity of
 Theorem \ref{teo-2} is established for symmetric potential and $f=I$. The arguments in \cite{fusco} extend naturally to cover the general case of $f$-equivariance. 

If $G$ is a finite reflection group we apply Theorem \ref{teo-2} to $u_R$ with $\Omega=D_{0,R}:=D_0\cap B_R$ and $Z=\{\gamma a\}_{\gamma\in\Gamma}\setminus\{a\}$. From Step.2 we have $u_R(\overline{F}_R)\subset\overline{\Phi}$ and therefore by $f$-equivariance
\begin{eqnarray}\label{d-inphi}
u_R(\overline{D}_{0,R})\subset u_R(\overline{D}_R)\subset\cup_{\gamma\in\Gamma_a}\gamma\overline{\Phi}=\overline{\mathcal{D}}.
\end{eqnarray}
Since by Hypothesis \ref{h3}, as we have seen, $a$ is the unique minimizer of $W$ in $\mathcal{D}$ it results
\begin{eqnarray}\label{delta-big}
u_R(D_{0,R})\cap Z_\delta=\varnothing,\;\text{ for }\;\delta=d(a,\partial\mathcal{D})>0.
\end{eqnarray}
This, the bound (\ref{below-emmeprime}) and Theorem \ref{teo-2} imply
\begin{eqnarray}\label{q-existence}
\vert u_R(x)-a\vert<q(d(x,\partial D_{0,R})),\;x\in D_{0,R},
\end{eqnarray}
for some positive function $r\rightarrow q(r)$ that does not depend on $R$. Therefore the
 inequality (\ref{q-existence}) passes to the limit along the sequence $u_{R_j}$ that defines
 the entire solution $u$. This and the $f$-equivariance of $u$ establish (ii). The exponential estimate (iii) follows from  Remark \ref{remark}. The proof of Theorem \ref{theorem1} is complete.

Assume now that $G$ is a discrete group. In this case the fundamental domain can be
 bounded as for the group $G^\prime$ considered in Section \ref{sec-2} or unbounded as for the group $G$ generated by the reflections in the plane $\{x_1=0\}$ and $\{x_1=1\}$
 of $\R^n,\,n>1$ where $F=\{x\in\R^n:0<x_1<1\}$. For establishing the estimates (ii) and (iii) in Theorem \ref{theorem2} it suffices to consider the case where $F$ is bounded. For $R>0$ define
\begin{eqnarray}\label{effeerre}
\left.\begin{array}{l}
v_R(y)=u_R(\frac{y}{R}),\;y\in\R^n,\\\\
F^R=\{y\in\R^n: \frac{y}{R}\in F\},
\end{array}\right.
\end{eqnarray}
and let $G^R$ the discrete reflection group generated by the reflections in the planes $P_1^R,\dots,P_l^R$ that bound $F^R$. There is an obvious group isomorphim
 $\eta^R:G^R\rightarrow G$ between $G^R$ and $G$ and the minimality of $u_R$ implies
 that $v_R\in W_{\rm loc}^{1,2}(\R^n,\R^m)$ is a local minimizers in the class of $f^R$-equivariant maps where $f^R:=f\circ\eta^R$. Therefore $v_R$ is a  solution of
\begin{eqnarray}\label{weak}
\Delta v-W_u(v)=0,\;\text{ in }\;\R^n.
\end{eqnarray}
This, (\ref{below-emme}) and elliptic regularity implies
\begin{eqnarray}\label{under-emmeprime}
\vert\nabla v_R\vert\leq M^\prime,\;\text{ in }\;\R^n,
\end{eqnarray}
for some $M^\prime>0$ independent of $R$. As before we have
\begin{eqnarray*}
u_R(\overline{D}_0)\subset\cup_{\gamma\in\Gamma_a}\gamma\overline{\Phi}=\overline{\mathcal{D}},
\end{eqnarray*}
or equivalently
\begin{eqnarray}\label{equvalently}
v_R(\overline{D}_0^R)\subset\overline{\mathcal{D}},
\end{eqnarray}
where $\overline{D}_0^R=\cup_{g\in G^a}(\eta^R)^{-1}(g)\overline{F}^R$. From (\ref{equvalently}) it follows
\begin{eqnarray}\label{delta-big1}
v_R(D_0^R)\cap Z_\delta=\varnothing,\;\text{ for }\;\delta=d(a,\partial\mathcal{D})>0.
\end{eqnarray}
Therefore we can apply Theorem \ref{teo-2} to $v_R$ with $\Omega=D_0^R$ and deduce, for $R\geq R_0:=r(\bar{q})$
\begin{eqnarray}\label{estimate-verre}
\vert v_R(y)-a\vert\leq q(d(y,\partial D_0^R)),\;y\in D_0^R,
\end{eqnarray}
which is equivalent to
\begin{eqnarray}\label{estimate-uerre}
\vert u_R(x)-a\vert\leq q(R d(x,\partial D_0)),\;x\in D_0.
\end{eqnarray}
The rest of the proof is as in the case of $G$ finite discussed before. The proof of Theorem \ref{theorem2} is complete.

\section{Three detailed examples involving the reflection group of the tetrahedron}
\subsection{Preliminaries}
In this subsection we recall some properties of the group of symmetry of a regular tetrahedron and the group of symmetry of a cube.

Let $\mathcal{T}$ be the group of symmetry of a regular tetrahedron $A_1A_2A_3A_4$ (with $A_1=(1,1,1)$, $A_2=(-1,-1,1)$, $A_3=(1,-1,-1)$, $A_4=(-1,1,-1)$) which can be
 inscribed in a cube centered at the origin $O$ (see Figure \ref{figure3}). The order of
 $\mathcal{T}$ is $|\mathcal{T}|=24$ and $\mathcal{T}$ is isomorphic to the permutation group $S_4$. The $24$ elements of $\mathcal{T}$ are associated to the following elements of $S_4$:
\begin{itemize}
\item $I_3$ the identity map of $\R^3$, corresponds to the unit of $S_4$ ($I_3$ fixes the $4$ vertices),
\item the $8$ rotations of angle $\pm 2 \pi/3$ with respect to the axes $OA_i$ correspond to the $3$-cycles $(j \ k \ l)$ (only the vertex $A_i$ is fixed),
\item the $6$ reflections with respect to the planes $OA_iA_j$ correspond to the transpositions $(k \ l)$ (the vertex $A_i$ and $A_j$ are fixed),
\item the $3$ symmetries with respect to the coordinate axes correspond to the permutations $(i \ j)(k \ l)$ (no vertex fixed),
\item a reflection with respect to the plane $OA_kA_l$ composed with a rotation with respect to the axis $OA_i$ corresponds to one of the six $4$-cycles $(i \ j \ k \ l)$ (no vertex fixed).
\end{itemize}
Since there exists a homomorphism between the permutation groups $S_4$ and $S_3$, and since $S_3$ is isomorphic to the dihedral group $D_3$ which is the group of symmetry of an equilateral triangle in the plane, we also have a homomorphism $\phi: \mathcal{T} \to D_3$. $\phi$ associates:
\begin{itemize}
\item $I_3$ and the symmetries with respect to the coordinate axes to $I_2$

(i.e. $\ker \: \phi = \{I_3 \text{ and the three symmetries with respect to the coordinate axes}\}$),
\item the rotation of angle $2 \pi/3$ with respect to the axis $OA_1$ to the rotation $\rho$ of angle $2 \pi /3$ in the plane,
\item the reflections with respect to the planes $OA_1A_2$ and $OA_3A_4$ to the reflection $\sigma_0:\R^2\rightarrow\R^2$ in the $u_1$  axis (cf. Figure \ref{figure3}),
\item the reflections with respect to the planes $OA_1A_4$ and $OA_2A_3$ to the reflection $\sigma_1=\rho \sigma_0$,
\item the reflections with respect to the planes $OA_1A_3$ and $OA_2A_4$ to the reflection $\sigma_2=\rho^2 \sigma_0$.
\end{itemize}

Let $\mathcal{K}$ be the group of symmetry of a cube centered at the origin $O$ with vertices at the points $(\pm 1,\pm 1, \pm1)$. The order of $\mathcal{K}$ is $|\mathcal{K}|=48$, and $\mathcal{K}$ contains the group of symmetry of the regular tetrahedron as a
 subgroup (i.e. $\mathcal{T}<\mathcal{K}$). Let $\sigma:\R^3 \to \R^3$, $\sigma (x)=-x$,
 be the antipodal map. Clearly, $\sigma$ is an element of $\mathcal{K}$ of order $2$ which does not belong to $\mathcal{T}$, and $\mathcal{K}=\mathcal{T} \cup \sigma
 \mathcal{T}$. Furthermore, $\sigma$ commutes with the reflections with respect to the planes $OA_iA_j$, thus $\sigma$ commutes with every element of $\mathcal{K}$. As a consequence, the correspondence
$$\{I_3,\sigma \} \times \mathcal{T} \ni(\alpha,\beta) \to \alpha\beta \in \mathcal{K},$$
defines an isomorphism of the group product $\{I_3,\sigma \}\times \mathcal{T}$ onto $\mathcal{K}$,
and we can define a homomorphism $\psi: \mathcal{K} \to \mathcal{T}$ by setting $\psi(\beta)=\beta$, and $\psi(\sigma \beta)=\beta$, for every $\beta \in \mathcal{T}$.
 By definition, $\psi$ leaves invariant the elements of $\mathcal{T}$. We also mention that
 $\mathcal{K}$ contains the $3$ reflections with respect to the coordinate planes $x_i=0$
 (which are the symmetries with respect to the coordinate axes $Ox_i$ composed with $\sigma$).

\subsection{A solution $u:\R^3 \to \R^2$ to \eqref{elliptic system} with the reflection group of the tetrahedron acting on the domain and the reflection group of the equilateral triangle acting on the target}

\

In this example we consider the aforementioned homomorphism $\phi:\mathcal{T} \to D_3$. Let $F$ be the fundamental domain of $\mathcal{T}$ bounded by the planes
 $OA_1A_2$, $OA_3A_4$ and $OA_1A_4$, and $\Phi$ be the fundamental domain of
 $D_3$ bounded by the lines $u_2=0$ and $u_2=\sqrt{3}u_1$ corresponding to the
 reflections $\sigma_0$ and $\sigma_1$. According to what precedes, the image by $\phi$
 of the reflections with respect to the planes $OA_1A_2$ and $OA_3A_4$ is $\sigma_0$,
 while the image by $\phi$ of the reflection with respect to the plane $OA_3A_4$ is
 $\sigma_1$. Thus $\phi$ is a positive homomorphism that associates $F$ to $\Phi$, and Hypothesis \ref{h1} is satisfied. If Hypotheses \ref{h2}--\ref{h3} also hold, Theorem \ref{theorem1} ensures the existence of a $\phi$-equivariant solution $u$ to \eqref{elliptic system}. In particular (see Figure \ref{figure3}) $\phi$-equivariance implies that $u$ maps
\begin{itemize}
\item the coordinate axes into the reflection lines (with the same colour),
\item the planes $OA_1A_2$ and $OA_3A_4$ into the reflection line $u_2=0$,
\item the planes $OA_1A_4$ and $OA_2A_3$ into the reflection line $u_2=\sqrt{3}u_1$,
\item the planes $OA_1A_3$ and $OA_2A_4$ into the reflection line $u_2=-\sqrt{3}u_1$,
\item the diagonals of the cube	at the origin $O$.
\end{itemize}
In addition Theorem \ref{theorem1} implies that the solution $u$ is positive (i.e. $u(\overline {F}) \subset \overline \Phi$), and
 therefore $u$ maps each fundamental domain of $\mathcal{T}$ into a fundamental domain
 of $D_3$ as in Figure \ref{figure3}. If for instance the potential $W$ has $6$ minima (one
 in the interior of each fundamental domain of $D_3$), then the domain space $\R^3$ is also split into six regions as in Figure \ref{figure3}. Properties (ii) and (iii) of Theorem
 \ref{theorem1} state that for every $x$ in such a region $D$, $u(x)$ converges as $d(x,\partial D) \to \infty$, to the corresponding minimum $a$ of $W$.

\begin{figure}[h]
\begin{center}
\includegraphics[scale=0.7]{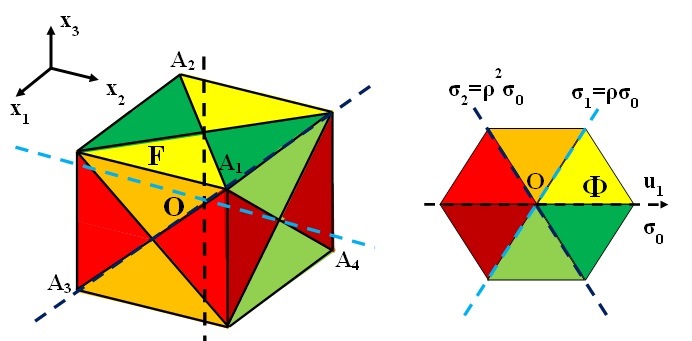}
\end{center}
\caption{Fundamental domains for the action of $\mathcal{T}$ on $\R^3$ (left) and for the action of $D_3$ on $\R^2$ (right). The $\phi$-equivariant
 solution $u:R^3\rightarrow\R^2$ of (\ref{elliptic system}) given by Theorem \ref{theorem1} maps fundamental domains into fundamental domains with the same color. In particular $u$ maps the infinite double cone (union of four fundamental domains) generated by $O$ and by the two yellow triangle into the sector $\Phi$.}
\label{figure3}
\end{figure}

\subsection{A solution $u:\R^3 \to \R^3$ to \eqref{elliptic system} with the reflection group of the cube acting on the domain and the reflection group of the tetrahedron acting on the target}

\

In this example we consider the homomorphism $\psi: \mathcal{K} \to \mathcal{T}$
 defined previously. Now, we denote by $F$ be the fundamental domain of $\mathcal{K}$
 bounded by the planes $OA_1A_2$, $OA_1A_4$ and $x_2=0$, and by $\Phi$ be the fundamental domain of $\mathcal{T}$ bounded by the planes $OA_1A_2$, $OA_1A_4$ and
 $OA_2A_3$ (cf. Figure \ref{figure4}). According to what precedes, $\psi$ leaves invariant
 the reflections with respect to the planes $OA_1A_2$ and $OA_1A_4$, while the image by
 $\psi$ of the reflection with respect to the plane $x_2=0$ is the coordinate axis $x_2$, that is, the intersection of the planes $OA_1A_4$ and $OA_2A_3$. Thus $\psi$ is a positive
 homomorphism that associates $F$ to $\Phi$, and Hypothesis \ref{h1} is satisfied. If
 Hypotheses \ref{h2}--\ref{h3} also hold, Theorem \ref{theorem1} ensures the existence of a $\psi$-equivariant solution $u$ to \eqref{elliptic system}. More precisely, $\psi$-equivariance implies that $u$ maps
\begin{itemize}
\item every plane $OA_iA_j$ into itself,
\item every diagonal of the cube into itself,
\item the coordinate planes into the perpendicular coordinate axes,
\item the coordinate axes at the origin $O$.
\end{itemize}
In addition, the solution $u$ is positive (i.e. $u(\overline {F}) \subset \overline \Phi$), and in fact $u$ maps each fundamental domain of $\mathcal{K}$ into a fundamental domain of $\mathcal{T}$ as in Figure \ref{figure4}.

\begin{figure}[h]
\begin{center}
\includegraphics[scale=0.7]{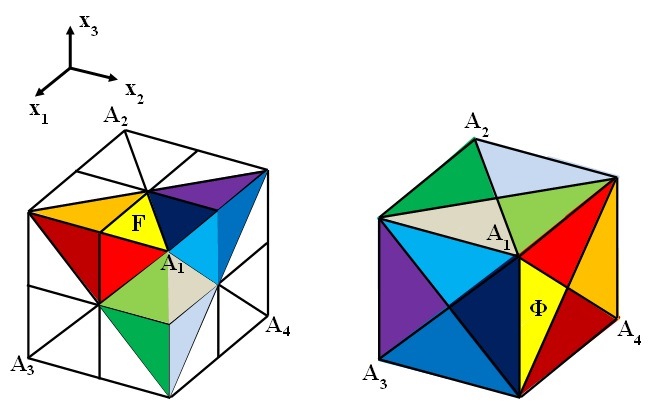}
\end{center}
\caption{Fundamental domains for the action on $\R^3$ of $\mathcal{K}$ (left) and $\mathcal{T}$ (right). The $\psi$-equivariant solution $u:\R^3\rightarrow\R^3$ of (\ref{elliptic system}) given by Theorem \ref{theorem1} maps fundamental domains into fundamental domains with the same color. Note in particular that $u$ maps $\overline{F}\cup\sigma\overline{F}$ into $\overline{\Phi}$.}
\label{figure4}
\end{figure}

\begin{figure}[h]
\begin{center}
\includegraphics[scale=0.7]{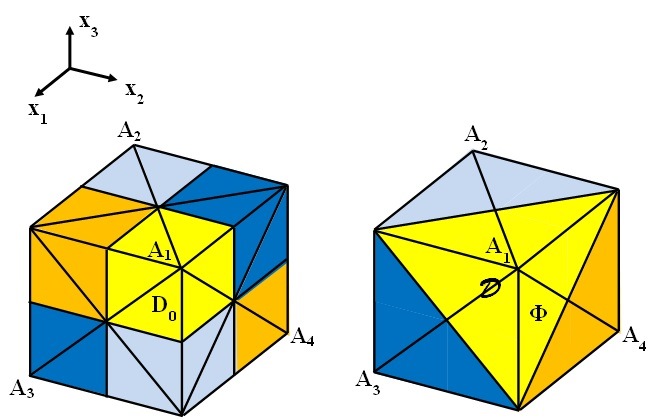}
\end{center}
\caption{The sets $\mathcal{D}=\{u=\alpha_1(1,1,-1)+\alpha_2(-1,1,1)+\alpha_3(1,-1,1), \alpha_i>0, i=1,2,3\}$, $D_0=\{x_1>0, \forall i=1,2,3\}$ and $D=D_0\cup\sigma D_0$ when $W$ has four minima on the vertices of the tetrahedron. In this case the solution $u$ of (\ref{elliptic system}) given by Theorem \ref{theorem1} satisfies: $u(x)\rightarrow A_1$, for $\min_i\vert x_i\vert\rightarrow+\infty,\;x\in D$.}
\label{figure5}
\end{figure}

If the potential $W$ has for instance $4$ minima (located at the vertices of the tetrahedron
 $A_1$,$A_2$,$A_3$ and $A_4$), then the stabilizer $\Gamma_a$ of $a=A_1$ in
 $\Gamma=\mathcal{T}$ has six elements: $I_3$, the reflections with respect to the planes $OA_1A_2$, $OA_1A_3$ and $OA_1A_4$, and the rotations of angle $\pm 2 \pi /3$ with
 respect to the axis $OA_1$. Thus $\mathcal{D}$ is the (interior of the closure) of the union of the six fundamental domains that have $A_1$ on their boundary; the group $\psi^{-1
}(\Gamma_a)=\Gamma_a \cup \sigma \Gamma_a$ has $12$ elements, and the set $D$
 has two connected components: the solid right angle $D_0=\{x_i>0, \: \forall i=1,2,3 \}$
 and $\sigma D_0=-D_0$ (cf. Figure \ref{figure5} and also note that the group $G^a$ is in
 this particular case the group $\Gamma_a$). According to Theorem \ref{theorem1}, if $x
 \in D_0$ and  $d(x,\partial D_0) \to \infty$ (that is, if $x_i \to +\infty$ for every
 $i=1,2,3$) then $u(x) \to a$. Of course, the same result is true when $x \in -D_0$ and
 $d(x,\partial (-D_0)) \to \infty$, and the solution also converges in the other solid right angle cones to the corresponding minima of $W$ as in Figure \ref{figure5}.

\subsection{A crystalline structure in $\R^3$}

\

Now, let us consider the discrete reflection group $\mathcal{K'}$ acting in $\R^3$ which is generated by the reflections $s_1$, $s_2$, $s_3$ and $s_4$ with respect to the corresponding planes $P_1:=OA_1A_2$, $P_2:=OA_1A_4$, $P_3:=\{x_2=0 \}$ and $P_4:=\{x_1+ x_3=2 \}$. These planes bound the fundamental domain $F'$ of $\mathcal{K'}$ with vertices at the points $O$, $A_1$, $I:=(1,0,1)$ and $B:=(0,0,2)$ (cf. Figure \ref{figure10}).
The point group of $\mathcal{K'}$, that is the stabilizer of the origin, is the group $\mathcal{K}$, and we have $\mathcal{K'}=T\mathcal{K}$, where $T$ is the translation
 group of $\mathcal{K'}$. $T$ is generated by the translations with respect to the vectors
 $t_1:=(2,0,2)$, $t_2:=(0,2,2)$ and $t_3:=(0,-2,2)$. By composing the canonical
 homomorphism $p: \mathcal{K'} \to \mathcal{K}$ such that $p(tg)=g$ for every $t \in T$
 and $g \in \mathcal{K}$, with the homomorphism $\psi: \mathcal{K} \to\mathcal{T}$ defined previously, we obtain a homomorphism $\psi': \mathcal{K'} \to \mathcal{T}$. We
 note that $\psi' (s_4)$ is the image by $\psi$ of the reflection with respect to the plane
 $OA_2A_3$, which is the reflection with  respect to $OA_2A_3$. Thus, $\psi'$ is a positive homomorphism which associates $F'$ to the fundamental domain $\Phi$ of $\mathcal{T}$ bounded by the planes $OA_1A_2$, $OA_1A_4$ and $OA_2A_3$.  In this case the
 elementary crystal $C=\cup_{g \in \mathcal{K}} \overline{gF'}$ is a {\it rhombic
 dodecahedron} (cf. Figure \ref{figure10}) that tiles the three dimensional space when
 translated by the elements of $T$ \footnote{Space filling tessellation with rhombic
 dodecahedra is the crystal structure in which often are found garnets and other minerals like pyrite and magnetite.}. Several structures are possible for the solution $u_R$ given by Theorem \ref{theorem2} depending on the position of $a\in\overline{\Phi}$. For instance if
 $a\in\Phi$ we have $\mathcal{D}=\Phi$ and $D_0=F^\prime$. If $a=(0,1,0)$ (cf. Figure
 \ref{figure11}) then $\mathcal{D}=\{u: \max\{\vert u_1\vert,\vert u_3\vert\}<u_2\}$ and
 $D_0$ is the pyramid with basis the rhombus defined by the points $A_1,B,(1,-1,1),(2,0,0)$, and vertex in $O$. Finally, if $a=A_1$, $\mathcal{D}$ is the cone which has vertex in $O$ and is generated by the triangle with vertices at the points
$(1,1,-1), (1,-1,1),(-1,1,1)$, while $D_0$ is the polyhedron (union of two pyramids) with vertices at the points $O, (2,0,0), (0,2,2), B, A_1$ and $D=\cup_{t\in T}(D_0\cup\sigma D_0)$.

\begin{figure}[h]
\begin{center}
\includegraphics[scale=.7]{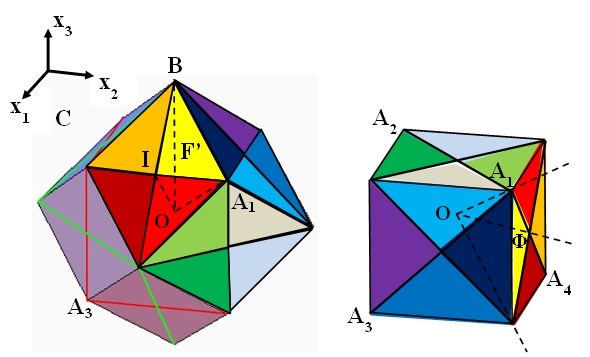}
\end{center}
 \caption{Fundamental domains for the action on $\R^3$ of $\mathcal{K'}$ (left) and $\mathcal{T}$ (right). The fundamental domain
 $F'$ of $\mathcal{K'}$ is a pyramid with basis the triangle $A_1BI$ and vertex in $O$. Under the action of the point group $\mathcal{K}$, $F'$ generates the rhombic dodecahedron $C$ (left) which tiles the domain space $\R^3$ when translated by the elements of $T$. The $\psi^\prime$-equivariant solution $u:\R^3\rightarrow\R^3$ of (\ref{elliptic system}) given by Theorem \ref{theorem2} maps fundamental domains into fundamental domains with the same color. Note in particular that $u$ maps $\cup_{t\in T}(\overline{F^\prime}\cup\sigma\overline{F^\prime})$ into $\overline{\Phi}$.}
\label{figure10}
\end{figure}

\begin{figure}[h]
\begin{center}
\includegraphics[scale=.7]{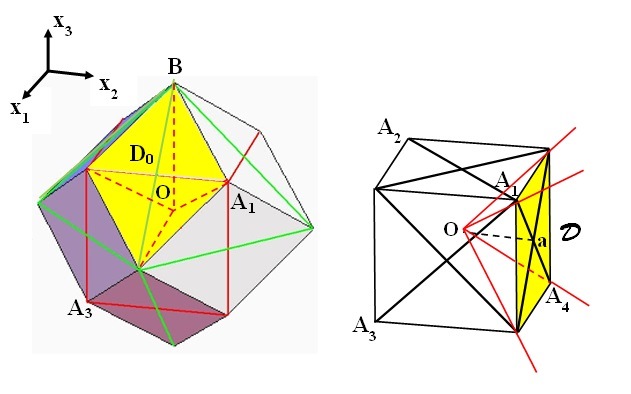}
\end{center}
\caption{The sets $D_0$ and $\mathcal{D}$ when $W$ has six minima (one in the middle of each side of the tetrahedron). In this case
the $\psi^\prime$-equivariant solution $u_R:\R^3\rightarrow\R^3$ given by Theorem \ref{theorem2} satisfies $\lim_{R\rightarrow+\infty}u_R(x+t)=(0,1,0)$ for $x\in D_0\cup\sigma D_0,\,t\in T$.}
\label{figure11}
\end{figure}

\section{Appendix: other examples in lower dimension}
\subsection{Positive homomorphisms between finite reflection groups of the plane}\label{sec:positive hom}

\

The finite reflection groups of the plane are the dihedral groups $D_n$ with $n \geq 1$. $D_n$ contains $2n$ elements: the rotations $r_n^0=I_2$, $r_n^1$,..., $r_n^{n-1}$ (where
 $I_2$ is the identity map of the plane, and $r_n$ is the rotation of angle $2 \pi/n$), and
 the reflections  $r_n^0 s=s$, $r_n^1 s$,..., $r_n^{n-1}s$ (where $s$ is the reflection with
 respect to the $x_1$ coordinate axis). Similarly, the elements of $D_{nk}$ (with $k \geq
 1$) are the rotations: $r_{nk}^0=I_2$, $r_{nk}^1$,..., $r_{nk}^{nk-1}$ (where $r_{nk}$ is
 the rotation of angle $2 \pi/nk$) and the reflections  $r_{nk}^0s=s$, $r_{nk}^1s$,...,
 $r_{nk}^{nk-1}s$. In the two Propositions below we have determined all the positive
 homomorphisms between finite reflection groups of the plane (up to an isomorphism).
 From the  list of homomorphisms between dihedral groups established in \cite{johnson}, we have extracted the positive ones.

For $m=\pm 1$ we define the homomorphism $f_m : D_{nk} \to D_n$, by setting $f_m(r_{nk}^p)= r_n^{mp}$ and  $f_m(r_{nk}^ps)= r_n^{mp}s$, for every integer $p$. Thanks to the property $sr=r^{-1}s$ (which holds for every reflection $s$ and every
 rotation $r$), it is easy to check that $f_m$ is a homomorphism from $D_{nk}$ onto
 $D_n$. We can also define the homomorphism $g: D_{2k} \to D_2$, by setting
 $g(r_{2k}^p)= s^p$ and $g(r_{2k}^ps)= s^p \sigma$ for every integer $p$, where $\sigma$ denotes the the antipodal map $\sigma u =-u$.

\begin{proposition}\label{proposition1}
If $n \geq 2$, $k \geq 1$, $G=D_{nk}$ acts on the domain plane $\R^2$, and $\Gamma=D_n$ on the target plane $\R^2$, then for every $m=\pm 1$, $f_m$ is a
 positive homomorphism which associates the fundamental domain $F:=\{re^{it} \mid 0<r,
 \: 0<t< \pi/nk \}$ of $G$ to the fundamental domain $\Phi:=\{re^{it} \mid 0<r, \:
 0<mt<\pi/n \}$ of $\Gamma$. In addition, the homomorphism $f_m$ leaves invariant the elements of $D_n<D_{nk}$ if and only if $mk=1 \mod n$.
\end{proposition}
\begin{proof} By definition, the lines that bound the fundamental domain $F$ correspond
 to the reflections $s$ and $r_{nk}s$. Since for $m=\pm 1$, the fixed points of the
 reflections $f_m(s)= s$ and $f_m(r_{nk}s)= r_n^{m}s$ are the lines that bound the fundamental domain $\Phi$, the homomorphism $f_m$ can associate $F$ to $\Phi$. $f_m$
 leaves invariant the elements of $D_n<D_{nk}$ if and only if $f_m(r_n)=r_n$, that is, if and
 only if
$$f_m(r_n)=f_m(r_{nk}^k)=r_n^{mk}=r_n \Leftrightarrow r_n^{mk-1}=I_2 \Leftrightarrow mk=1 \mod n.$$
\end{proof}
\begin{proposition}\label{proposition2}
If $k \geq 1$, $G=D_{2k}$ acts on the domain plane $\R^2$, and $\Gamma=D_2$ on the
 target plane $\R^2$, then $g$ is a positive homomorphism which associates the
 fundamental domain $F=\{re^{it} \mid 0<r, \: 0<t< \pi/2k \}$ of $G$ to the fundamental domain
$\Phi=\{re^{it} \mid 0<r, \: 0<t<\pi/2 \}$ of $\Gamma$.
\end{proposition}
\begin{proof}
As previously, we see that $g(s)= \sigma$ fixes only the origin, while $g(r_{2k}s)= s \sigma$ fixes the $u_2$ coordinate axis. Thus, the homomorphism $g$ can associate $F$
 to $\Phi$ (and in fact it can associate $F$ to any of the $4$ fundamental domains of $D_2$).
\end{proof}
To illustrate the Propositions above, let us give some examples.
\begin{itemize}
\item The homomorphism $f:D_6 \to D_3$ that was mentioned at the beginning of Sec.2 (cf. also \cite{bates-fusco}) coincides with the homomorphism $f_{-1}$ of Proposition
 \ref{proposition1} with $n=3$, $k=2$ and $m=-1$. Since $mk=-2=1 \mod 3$, we see
 again that it leaves invariant the elements of $D_3$.
\item Taking $n=3$, $k=5$ and $m=-1$, we check that $mk=1 \mod n$, and we obtain a
 new homomorphism $f_{-1}: D_{15} \to D_3$ that leaves invariant the elements of $D_3$. The kernel of this homomorphism is the cyclic group generated by the rotation $r_5$.
\item Taking $n=2$, $k=2$ and $m=1$, we obtain the positive homomorphism $f_1:D_4 \to D_2$. When Hypotheses \ref{h2}--\ref{h3} also hold, Theorem \ref{theorem1} ensure
 the existence of a $f_1$-equivariant solution to \eqref{elliptic system} which maps
 $\overline F$ into $\overline \Phi$, and the other fundamental domains of $G=D_4$ as in Figure \ref{figure6}.
\item Considering the homomorphism $g:D_4 \to D_2$ of Proposition \ref{proposition2} (with k=2) we can also construct a g-equivariant solution $u$ to \eqref{elliptic system}.
 This solution has the particularity that the coordinate axes are mapped at the origin.
 Indeed, by $g$-equivariance, if $x,y \in \R^2$ are symmetric with respect to one of the coordinate axes, then $u(x)=-u(y)$ (cf. Figure \ref{figure7} for the correspondence of the fundamental domains).
\end{itemize}

\begin{figure}[h]
\begin{center}
\includegraphics[scale=0.7]{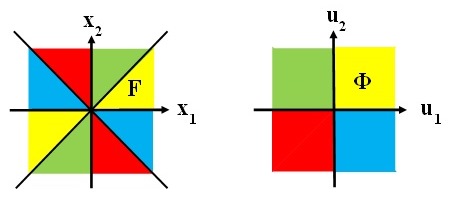}
\end{center}
\caption{The correspondence of the fundamental domains for a solution to \eqref{elliptic system} equivariant with respect to the homomorphism $f_1:D_4 \to D_2$.}
\label{figure6}
\end{figure}

\begin{figure}[h]
\begin{center}
\includegraphics[scale=0.7]{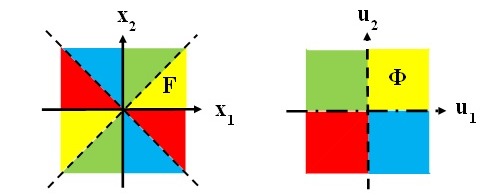}
\end{center}
\caption{The correspondence of the fundamental domains for a solution to \eqref{elliptic system} equivariant with respect to the homomorphism $g:D_4 \to D_2$.}
\label{figure7}
\end{figure}

\subsection{Saddle solutions}

\

In this subsection we are going to construct as an application of Theorems \ref{theorem1} and \ref{theorem2}, scalar solutions to \eqref{elliptic system} and \eqref{elliptic system R}
 having particular symmetries. The only finite reflection group that acts on the target space
 $\R$ is the dihedral group $\Gamma=D_1$ with two elements: the identity $I_1$ and the
 antipodal map $\sigma u=-u$. We assume that a finite or a discrete reflection group $G$
 acts on the domain space $\R^n$, and that $W:\R \to \R$ satisfies Hypotheses \ref{h2}--\ref{h3}, that is,
\begin{itemize}
\item $W$ is a nonnegative and even function,
\item there exists $M>0$ such that $W(u) \geq W(M)$, for $u \geq M$,
\item $W(u)=0 \Leftrightarrow u=\pm a$, with in addition $a>0$ and $W''(a)>0$.
\end{itemize}
Clearly, the map $\epsilon$ which sends the orientation-preserving motions to $I_1$, and the orientation-reversing motions to $\sigma$, is a positive homomorphism from $G$ onto $\Gamma$. Thus, Theorems \ref{theorem1} and \ref{theorem2} ensure the existence of classical solutions $u: \R^n \to \R$ to \eqref{elliptic system} and \eqref{elliptic system R} with the following properties:
\begin{itemize}
\item[(i)] $\epsilon$-equivariance implies that if $x,y \in \R^n$ are symmetric with respect to a reflection plane of $G$, then $u(x)=-u(y)$. In particular, $u$ vanishes on the reflection planes of $G$. If $G$ is a discrete reflection group, then $u$ is periodic in the sense that $u(x+t)=u(x)$, for every $x \in \R^n$, and every translation $t$ in the translation group $T<G$.
\item[(ii)] Positivity means that either $u \geq 0$ or $u \leq 0$ in each fundamental domain $F$ of $G$.
\item[(iii)] In each fundamental domain $F$, $u(x)$ approaches either $a$ or $-a$, as $x \in F$ and $d(x,\partial F)$ increases.
\end{itemize}

We also give another example, more elaborated, when $G=D_{2k}$ (with $k \geq 1$) acts on $\R^2$. Let us consider the homomorphism
$h: D_{2k} \to D_1$ such that $h(r_{2k}^p)= \sigma^p$ and $h(r_{2k}^ps)= \sigma^p$, for every integer $p$ (see the previous subsection for the notation).
In this particular set-up, we can again construct a $h$-equivariant solution $u: \R^2 \to \R$ to \eqref{elliptic system} which has in each fundamental domain of $D_{2k}$ alternatively even and odd symmetries. Figure \ref{figure8} represents the symmetries of such a solution for $k=2$.

\begin{figure}[h]
\begin{center}
\includegraphics[scale=0.7]{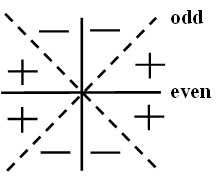}
\end{center}
\caption{The symmetries of a solution $u: \R^2 \to \R$ to \eqref{elliptic system} equivariant with respect to the homomorphism $h:D_4 \to D_1$.}
\label{figure8}
\end{figure}

\subsection{Other examples involving discrete reflection groups}

\

To finish, we give some more examples illustrating Theorem \ref{theorem2}. Let us assume again that the discrete reflection group $G'$ acts on the domain $x$-plane as in Sec.2 (cf. also the end of Sec.3), but let us consider now a new reflection group acting on the target
 $u$-plane: the group $\Gamma=D_2$. We construct a homomorphism $f'': G' \to D_2$ by
 composing the canonical projection $p:G' \to D_6$ with the homomorphism $g:D_6 \to
 D_2$ defined in subsection \ref{sec:positive hom} (that is, $f''=g \circ p$). As we did
 before for the homomorphism $f'$, we can check that $f''$ is a positive homomorphism.
 Thus, once again, Theorem \ref{theorem2} allows us to construct $f''$-equivariant
 solutions $u_R$ to \eqref{elliptic system R}. Figure \ref{figure9} represents the
 correspondence of the fundamental domains of $G'$ with the fundamental domains of
 $D_2$ for such solutions (compare with Figure \ref{figure2}). The $f''$-equivariant solutions have the particularity that some reflection lines of the group $G'$ are mapped at the origin.

\begin{figure}[h]
\begin{center}
\includegraphics[scale=0.7]{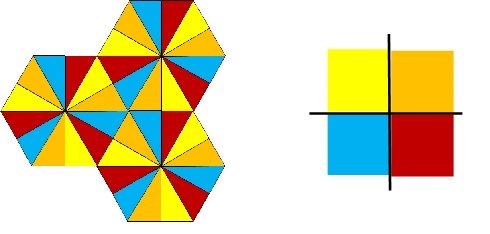}
\end{center}
\caption{The correspondence of the fundamental domains by the $f''$-equivariant solution $u_R$: on the right the fundamental domains of the discrete reflection group $G'$ and on the left the fundamental domains of the finite reflection group $D_2$.}
\label{figure9}
\end{figure}

Let us also mention a last example involving the discrete reflection group of the plane $H$, generated by the reflections with respect to the lines $x_2=0$, $x_2=x_1$ and $x_1=1$. The point group associated to $H$ is the group $D_4$, and we can compose the canonical projection $H \to D_4$ either with the homomorphism $f_1:D_4 \to D_2$ or with the homomorphism $g:D_4 \to D_2$ defined in subsection \ref{sec:positive hom}, to construct positive homomorphisms from $H$ onto $D_2$.

\

\

\

\

\

\nocite{*}
\bibliographystyle{plain}

\end{document}